\documentclass{amsart}
\usepackage{xypic}
\usepackage{BOONDOX-cal}

\def\spcheck{^\vee}

\newcommand{\VR}{\mathcal{O}}
\newcommand{\invo}{\overline{\rule{2.5mm}{0mm}\rule{0mm}{4pt}}}
\newcommand{\Proj}{\mathbb{P}^1_K}
\newcommand{\E}{\mathcal{E}}
\newcommand{\B}{\mathcal{B}}
\newcommand{\Herm}{\mathcal{H}}
\newcommand{\Q}{\mathcal{Q}}

\newcommand{\Indec}{\mathcal{I}}
\newcommand{\M}{\mathcal{M}}
\newcommand{\N}{\mathcal{N}}
\newcommand{\T}{\mathcal{T}}
\newcommand{\Id}{\operatorname{Id}}
\newcommand{\ovb}{q}
\DeclareMathOperator{\rk}{rk}
\DeclareMathOperator{\GL}{GL}
\DeclareMathOperator{\charac}{char}
\DeclareMathOperator{\Trd}{Trd}
\DeclareMathOperator{\Hom}{Hom}
\DeclareMathOperator{\End}{End}
\DeclareMathOperator{\diag}{diag}
\DeclareMathOperator{\image}{im}
\DeclareMathOperator{\Alt}{Alt}

\DeclareMathOperator{\Sym}{Sym}

\DeclareMathOperator{\coker}{coker}
\DeclareMathOperator{\Spec}{Spec}
\DeclareMathOperator{\ind}{ind}
\DeclareMathOperator{\sw}{\mathsf{sw}}
\DeclareMathOperator{\calEnd}{\mathcal{E\!n\!d}}

\newtheorem{lemma}{Lemma}[section]
\newtheorem{prop}[lemma]{Proposition}
\newtheorem{cor}[lemma]{Corollary}
\newtheorem{thm}[lemma]{Theorem}

\theoremstyle{definition}
\newtheorem{defis}[lemma]{Definitions}
\newtheorem{definition}[lemma]{Definition}

\theoremstyle{remark}
\newtheorem*{remark*}{Remark}
\newtheorem{example}[lemma]{Example}
\newtheorem{remark}[lemma]{Remark}

\title{Excellence of function fields of conics} 
\author{A.S. Merkurjev \and J.-P. Tignol}
\address{Department of Mathematics\\
University of California at Los Angeles\\
Los Angeles, California, 90095-1555}
\email{merkurev@math.ucla.edu}
\thanks{The work of the first author has been supported by the NSF
  grant DMS \#1160206.}
\address{ICTEAM Institute\\
Universit\'e catholique de Louvain\\
B-1348 Louvain-la-Neuve, Belgium}
\email{Jean-Pierre.Tignol@uclouvain.be}
\thanks{The second author acknowledges support from the Fonds de la
  Recherche Scientifique--FNRS under grant n$^\circ$~J.0014.15.}
\date{\today}

\begin{document}
\maketitle

A field extension $L/F$ is said to be excellent if for every quadratic
form $q$ over $F$ the anisotropic kernel of the form $q_L$ obtained
from $q$ by scalar extension to $L$ is defined over
$F$. Arason~\cite{Ara} first noticed that function fields of smooth
projective conics have this useful property. As it relies on
Knebusch's Habilitationschrift~\cite{Kn} on symmetric bilinear forms,
Arason's proof 
requires\footnote{Arason's proof can readily be extended to symmetric
  bilinear forms in characteristic~$2$, but this case is uninteresting
  because anisotropic bilinear forms in characteristic~$2$ remain
  anisotropic over the function field of a smooth projective conic by
  \cite[Cor.~3.3]{Lag}.} the hypothesis that $\charac F\neq2$.

Three other proofs have been published; they are due to
Rost~\cite[Corollary]{Rost}, Parimala~\cite[Lemma~3.1]{CTS},
\cite[Proposition~2.1]{PSS}, and Pfister~\cite[Prop.~4]{P}. Pfister's
proof is based on the study of quadratic lattices over the ring of an
affine open set of the conic, while Rost's proof uses ingenious
manipulations of quadratic forms that are isotropic over the function
field. Parimala's proof relies, like Arason's, on vector
bundles over the conic, but it uses the Riemann--Roch theorem instead
of Grothendieck's classification of vector bundles over the projective
line~\cite{Groth}. (Another unpublished proof was obtained by Van Geel
\cite{VG} as 
an application of the Riemann--Roch theorem.) The version of Parimala's
proof in \cite{PSS} has the extra feature to apply to hermitian forms
over division algebras instead of just quadratic forms, but all the
proofs published so far require $\charac F\neq2$.

Our goal in this paper is to prove the excellence of function fields
of smooth projective conics in arbitrary characteristic for hermitian
forms and generalized quadratic forms over division algebras. Our
proof is close in spirit to Arason's original proof: the idea is to
show that the anisotropic kernel of a hermitian or generalized
quadratic form extended to $L$ is the generic fiber of a nondegenerate
hermitian or generalized quadratic form on a vector bundle over the
conic. We then use the classification of these vector bundles to
conclude that the anisotropic kernel is extended from~$F$. Our
approach is completely free of any assumption on the characteristic of
the base field. Therefore, the case of generalized quadratic forms
requires a separate treatment, which is more delicate.

In \S\ref{sec:qf} we revisit the notion of quadratic form as defined
by Tits in \cite{Tits68}. Our goal is to rephrase Tits's definition 
in terms of modules over central simple algebras instead of vector
spaces over division algebras. We thus obtain a notion that is better
behaved under scalar extension. Hermitian forms and generalized
quadratic forms on vector bundles over a conic are discussed in
\S\ref{sec:qfbndl}, and the proof of the excellence result is given
in~\S\ref{sec:excel}.
To make our exposition as elementary as possible, we thoroughly
discuss in an appendix the classification of vector bundles over
smooth projective conics, using a representation of these bundles as
triples consisting of their generic fiber, the stalk at a closed point
$\infty$, and their section over the complement of $\infty$. Thus, we
give an elementary proof of Grothendieck's classification theorem, and
correct Arason's misleading statement\footnote{``Now the proof of the first
  sentence of \cite[Theorem~13.2.2]{Kn} (and the result of
  \cite{Groth} which is cited there) only depends on the projective
  line being a complete regular irreducible curve of genus zero''
  \cite{Ara}.} suggesting that vector bundles over a conic decompose
into line bundles.

We use the following notation throughout: for every linear
endomorphism $\sigma$ such that $\sigma^2=\Id$, we let
\[
  \Sym(\sigma)  = \ker(\Id-\sigma)\quad\text{and}\quad
  \Alt(\sigma)  = \image(\Id-\sigma).
\]
Thus, $\Alt(\sigma)\subset\Sym(-\sigma)$ always, and 
$\Alt(\sigma)=\Sym(-\sigma)$ in characteristic
different from~$2$.

\section{Quadratic forms}
\label{sec:qf}

\subsection{The definition}

Let $A$ be a central simple algebra over an arbitrary field $F$, and
let $\sigma$ be an $F$-linear involution on $A$. Let $M$ be a finitely
generated right $A$-module. The dual module $M^*=\Hom_A(M,A)$ has a
left $A$-module structure given by $(af)(x)=af(x)$ for $a\in A$, $f\in
M^*$, and $x\in M$. Let ${}^\sigma M^*$ be the right $A$-module
defined by
\[
{}^\sigma M^*=\{{}^\sigma f\mid f\in M^*\}
\]
with the operations
\[
{}^\sigma f+{}^\sigma g={}^\sigma(f+g) \quad\text{and}\quad {}^\sigma
f\cdot a ={}^\sigma(\sigma(a)f)
\]
for $a\in A$ and $f$, $g\in M^*$. Identifying ${}^\sigma f$ with the
map $x\mapsto \sigma\bigl(f(x)\bigr)$, we may also consider ${}^\sigma
M^*$ as the $A$-module of additive maps $g\colon M\to A$ such that
$g(xa) = \sigma(a)g(x)$ for $x\in M$ and $a\in A$, i.e., ${}^\sigma
M^*$ is the $A$-module of \emph{$\sigma$-semilinear maps} from $M$ to
$A$. 

Let $B(M)$ be the $F$-space of
sesquilinear forms $M\times M\to A$. Mapping ${}^\sigma f\otimes g$ to
the sesquilinear form $(x,y)\mapsto \sigma(f(x))g(y)$ defines a
canonical isomorphism
\[
{}^\sigma M^*\otimes_A M^*=B(M).
\]
Let $\sw\colon B(M)\to B(M)$ be the $F$-linear map taking a form
$b$ to the form $\sw(b)$ defined by
\[
\sw(b)(x,y)=\sigma\bigl(b(y,x)\bigr).
\]
Thus, $\sw({}^\sigma f\otimes g) = {}^\sigma g\otimes f$ for $f$,
$g\in M^*$. 

\begin{defis}
The space of (generalized) quadratic forms on $M$ is the factor space
\[
Q(M)=B(M)/\Alt(\varepsilon\sw),
\]
where $\varepsilon=1$ if $\sigma$ is orthogonal and $\varepsilon=-1$ if
$\sigma$ is symplectic. For $\delta=\pm1$,
the space of $\delta$-hermitian forms on $M$ is 
\[
H_\delta(M)=\Sym(\delta\sw)\subset B(M).
\]
\end{defis}

To relate this definition of quadratic form to the one given by Tits
in \cite{Tits68}, 
note that $B(M)$ is a free right module of rank~$1$ over $\End_AM$,
for the scalar multiplication defined as follows: for $b\in B(M)$ and
$\varphi\in\End_AM$, 
\begin{equation*}
(b\cdot\varphi)(x,y)=b(x,\varphi(y))\qquad\text{for $x$, $y\in M$.}
\end{equation*}
The pair $(B(M),\varepsilon\sw)$ is a \textit{space of bilinear
  forms for $\End_AM$}, in the sense of \cite[2.1]{Tits68}. With this
choice of space of bilinear forms, the elements of $Q(M)$ as defined
above are exactly the quadratic forms defined in \cite[2.2]{Tits68}.

By definition, the vector spaces $H_\varepsilon(M)$ and $Q(M)$ fit
into the exact sequence
\[
0\to H_\varepsilon(M)\to B(M)\xrightarrow{\Id-\varepsilon\sw} B(M) \to
Q(M)\to 0.
\]
Since $(\Id+\varepsilon\sw)\circ(\Id-\varepsilon\sw)=0$, there is a
canonical ``hermitianization'' map 
\[
\beta\colon Q(M)\to H_\varepsilon(M),
\] 
which associates to each quadratic form
$\ovb=b+\Alt(\varepsilon\sw)$ the $\varepsilon$-hermitian form
\[
\beta(\ovb)=b+\varepsilon\sw(b).
\] 
Thus, by definition the form $\beta(\ovb)$ actually lies in
$\Alt(-\varepsilon\sw)\subset H_\varepsilon(M)$. 

\subsection{Relation with submodules}

For every submodule $N\subset M$, the following exact sequence splits:
\begin{equation}
\label{eq:exseq}
0\to N\to M\to M/N\to 0.
\end{equation}
It yields by duality the split exact sequence
\[
0\to(M/N)^*\to M^*\to N^*\to 0,
\]
which allows us to identify $(M/N)^*$ with the submodule of linear
forms in $M^*$ that vanish on $N$. We thus obtain a canonical split
injective map
\[
B(M/N)={}^\sigma(M/N)^*\otimes_A(M/N)^* \to {}^\sigma M^*\otimes_AM^*
= B(M)
\]
and a canonical split surjective map
\[
B(M)={}^\sigma M^*\otimes_AM^*\to {}^\sigma N^*\otimes_AN^* = B(N).
\]
These canonical maps commute with $\Id-\delta\sw$ for $\delta=\pm1$,
hence they induce canonical maps
\[
H_\delta(M/N)\to H_\delta(M),\qquad H_\delta(M)\to H_\delta(N)\qquad
\text{for $\delta=\pm1$},
\]
and
\[
Q(M/N)\to Q(M),\qquad Q(M)\to Q(N).
\]

\begin{remark}
  \label{rem:Qinj}
  For a fixed splitting of the exact sequence~\eqref{eq:exseq}, the
  corresponding splittings of the injection $B(M/N)\to B(M)$ and the
  surjection $B(M)\to B(N)$ also commute with $\Id-\varepsilon\sw$,
  hence the map $Q(M/N)\to Q(M)$ is split injective and $Q(M)\to Q(N)$
  is split surjective.
\end{remark}

\begin{prop}
  \label{prop:sublagred}
  The canonical embedding $B(M/N)\to B(M)$ identifies $B(M/N)$ with
  the space of sesquilinear forms $b\in B(M)$ such that
  $b(x,y)=b(y,x)=0$ for all $x\in M$ and $y\in N$.
\end{prop}

\begin{proof}
  It is clear from the definition that the sesquilinear forms in the
  image of $B(M/N)$ vanish in ${}^\sigma M^*\otimes_AN^*$ and in
  ${}^\sigma N^*\otimes_AM^*$, hence they satisfy the stated property.

  For the converse, we use the canonical isomorphism
  \begin{equation}
  \label{eq:canisoB}
  {}^\sigma M^*\otimes_AM^*=\Hom_A(M,{}^\sigma M^*)
  \end{equation}
  mapping ${}^\sigma f\otimes g$ to the homomorphism
  $x\mapsto{}^\sigma f\cdot g(x)$. This isomorphism identifies each 
  sesquilinear form $b\in B(M)$ with the homomorphism
  $\widehat b\colon M\to{}^\sigma M^*$ mapping $x\in M$ to
  $b(\bullet,x)$. If $b(x,y)=b(y,x)=0$ for $x\in M$ and $y\in N$, then
  the image of $\widehat b$ lies in ${}^\sigma(M/N)^*$ and its kernel
  contains $N$. Therefore, $\widehat b$ induces a homomorphism $M/N\to
  {}^\sigma(M/N)^*$, and $b$ is the image of the corresponding
  sesquilinear form in $B(M/N)$.
\end{proof}

\subsection{Sublagrangian reduction of hermitian forms}

Let $\delta=\pm1$. For $h\in H_\delta(M)$ and $N\subset M$ any
$A$-submodule, we 
define the \emph{orthogonal} $N^\perp$ of $N$ by
\[
N^\perp=\{x\in M\mid h(x,y)=0\text{ for all $y\in N$}\}.
\]
The submodule $N$ is said to be a \emph{sublagrangian}, or a
\emph{totally isotropic submodule} of $M$, if $N\subset N^\perp$ or,
equivalently, if $h$ lies in the kernel of the restriction map
$H_\delta(M)\to H_\delta(N)$. The form $h$ is said to be
\emph{isotropic} if $M$ contains a nonzero sublagrangian. It is said
to be \emph{nonsingular} if the corresponding map $\widehat h\colon
M\to {}^\sigma M^*$ under the isomorphism~\eqref{eq:canisoB} is
bijective.

\begin{prop}
  \label{prop:sublagredherm}
  Let $h\in H_\delta(M)$ and let $N\subset M$ be a
  sublagrangian. There is a unique form $h_0\in H_\delta(N^\perp/N)$
  that maps under the canonical map $H_\delta(N^\perp/N)\to
  H_\delta(N^\perp)$ to the restriction of $h$ to $N^\perp$. The form
  $h_0$ is nonsingular if $h$ is nonsingular; it is anisotropic if $N$
  is a maximal sublagrangian.
\end{prop}

\begin{proof}
  The existence of $h_0$ readily follows from
  Proposition~\ref{prop:sublagred}. The form $h_0$ is unique because
  the map $B(N^\perp/N)\to B(N^\perp)$ is injective.

  Now, assume $h$ is nonsingular. Since $\widehat h$ carries $N^\perp$
  to ${}^\sigma(M/N)^*$, there is a commutative diagram with exact
  rows:
  \[
  \xymatrix{
  0\ar[r]&N^\perp\ar[r]\ar[d]_{\varphi} & M\ar[r]\ar[d]_{\widehat h} &
  M/N^\perp\ar[r]\ar[d]^{\psi}&0\\
  0\ar[r]& {}^\sigma(M/N)^*\ar[r]& {}^\sigma M^*\ar[r]&
  {}^\sigma N^*\ar[r] &0
  }
  \]
  The map $\psi$ is injective by definition of $N^\perp$, and
  $\widehat h$ is bijective because $h$ is nonsingular, hence
  $\varphi$ is an isomorphism. By duality, $\varphi$ yields an
  isomorphism ${}^\sigma\varphi^*\colon M/N\to
  {}^\sigma(N^\perp)^*$. Composing $\varphi$ with the inclusion
  ${}^\sigma(M/N)^*\subset {}^\sigma M^*$ and ${}^\sigma\varphi^*$
  with the canonical map $M\to M/N$, we obtain maps $\varphi'$,
  $\varphi''$ that fit into the following diagram with exact rows,
  where $i$ is the inclusion:
  \[
  \xymatrix{
  0\ar[r]&N\ar[r]\ar[d]_{i}& M\ar[r]^-{\varphi''}\ar[d]_{\widehat h} &
  {}^\sigma(N^\perp)^*\ar[r]\ar[d]^{{}^\sigma i^*} & 0\\
  0\ar[r]&N^\perp\ar[r]^{\varphi'}& {}^\sigma M^*\ar[r] & {}^\sigma
  N^*\ar[r]& 0
  }
  \]
  Since $\widehat h$ is bijective, the Snake Lemma yields an
  isomorphism ${}^\sigma(N^\perp/N)^*\stackrel{\sim}{\to}
  N^\perp/N$. Computation shows that the inverse of this isomorphism,
  viewed in $B(N^\perp/N)$, is $\sw(h_0)=\delta h_0$. Therefore, $h_0$
  is nonsingular.

  If $L\subset N^\perp/N$ is a sublagrangian for $h_0$, then the
  inverse image $L'\subset N^\perp$ of $L$ under the canonical map
  $N^\perp\to N^\perp/N$ is a sublagrangian for $h$. Therefore, $h_0$
  is anisotropic if $N$ is a maximal sublagrangian.
\end{proof}

When $N$ is a maximal sublagrangian, the anisotropic form $h_0$ is
called an \emph{aniso\-tropic kernel} of $h$. As for quadratic forms (see
Proposition~\ref{prop:uniquanisotker} below), the anisotropic kernel
of a $\delta$-hermitian form is uniquely determined up to isometry.

\subsection{Sublagrangian reduction of quadratic forms}

We say that a quadratic form $\ovb\in Q(M)$
is \emph{nonsingular} if its hermitianized form $\beta(\ovb)$
is nonsingular.\footnote{In \cite{Tits68}, Tits defines non-degenerate
  quadratic forms by a less stringent condition.} The form
$\ovb$ is said to be \emph{isotropic} if there exists a
nonzero submodule $N\subset M$ such that $\ovb$ lies in the kernel of
the restriction map $Q(M)\to Q(N)$; the
submodule $N$ is then said to be \emph{totally isotropic} for
$\ovb$. Clearly, any totally isotropic submodule $N$ for $\ovb$ is
also totally isotropic for the hermitianized form
$\beta(\ovb)$, hence it lies in its orthogonal $N^\perp$ for
$\beta(\ovb)$. 

\begin{prop}
  \label{prop:sublagredquad}
  Let $\ovb\in Q(M)$ and let $N\subset M$ be a
  totally isotropic submodule. There is a unique form $\ovb_0\in
  Q(N^\perp/N)$ that maps under the canonical map $Q(N^\perp/N)\to
  Q(N^\perp)$ to the restriction of $\ovb$ to $N^\perp$. The form
  $\ovb_0$ is nonsingular if $\ovb$ is nonsingular; it is anisotropic
  if $N$ is a maximal totally isotropic submodule.  
\end{prop}

\begin{proof}
  Let $b\in B(M)$ be a sesquilinear form such that
  $\ovb=b+\Alt(\varepsilon\sw)$. Since $N$ is totally isotropic for
  $\ovb$, there is a form $c\in B(M)$ such that
  \begin{equation}
    \label{eq:sublagredquad1}
    b(x,y)=c(x,y)-\varepsilon\sigma\bigl(c(y,x)\bigr) \qquad\text{for
      all $x$, $y\in N$.}
  \end{equation}
  Because $N^\perp/N$ is a projective module, there is a homomorphism
  $\pi\colon N^\perp\to N$ that splits the inclusion $N\hookrightarrow
  N^\perp$. Define a sesquilinear form $b_1\in B(N^\perp)$ by
  \[
  b_1(x,y)=b\bigl(x,\pi(y)\bigr) - c\bigl(\pi(x),\pi(y)\bigr)
  \qquad\text{for $x$, $y\in N^\perp$.}
  \]
  For $x\in N$ and $y\in N^\perp$, we have
  \begin{align}
    \label{eq:sublagredquad2}
    b(x,y)-b_1(x,y)+\varepsilon\sigma\bigl(b_1(y,x)\bigr) =&\; b(x,y) -
    b\bigl(x,\pi(y)\bigr) + c\bigl(\pi(x),\pi(y)\bigr)\\ & +
    \varepsilon\sigma\bigl(b(y,\pi(x))-c(\pi(y),\pi(x))\bigr). \notag
  \end{align}
  Since $\pi(x)=x$, \eqref{eq:sublagredquad1} yields
  \[
  b\bigl(x,\pi(y)\bigr) = c\bigl(\pi(x),\pi(y)\bigr) -
  \varepsilon\sigma\bigl(c(\pi(y),\pi(x))\bigr),
  \]
  hence three terms cancel on the right side
  of~\eqref{eq:sublagredquad2}, and we have
  \begin{equation}
    \label{eq:sublagredquad3}
    b(x,y)-b_1(x,y)+\varepsilon\sigma\bigl(b_1(y,x)\bigr) = b(x,y)
    +\varepsilon\sigma\bigl(b(y,x)\bigr) = \beta(\ovb)(x,y)=0.
  \end{equation}
  Similarly, for $x\in N$ and $y\in N^\perp$ we have
  \[
  b(y,x)=-\varepsilon\sigma\bigl(b(x,y)\bigr)
  \]
  hence \eqref{eq:sublagredquad3} yields
  \[
  b(y,x)-b_1(y,x)+\varepsilon\sigma\bigl(b_1(x,y)\bigr)=0.
  \]
  Therefore, letting $b\rvert_{N^\perp}$ denote the restriction of $b$
  to $N^\perp$, we may apply Proposition~\ref{prop:sublagred} to get a
  sesquilinear form $b_0\in B(N^\perp/N)$ that maps to
  $b\rvert_{N^\perp}-(\Id-\varepsilon\sw)(b_1)$ in $B(N^\perp)$. Then
  the quadratic form $\ovb_0=b_0+\Alt(\varepsilon\sw)\in Q(N^\perp/N)$
  maps to $\ovb\rvert_{N^\perp}$ in $Q(N^\perp)$. Uniqueness of the
  form $\ovb_0$ is clear since the map $Q(N^\perp/N)\to Q(N^\perp)$ is
  injective (see Remark~\ref{rem:Qinj}).

  Since $N$ is totally isotropic for the hermitianized form
  $\beta(\ovb)\in H_\varepsilon(M)$,
  Proposition~\ref{prop:sublagredherm} yields an
  $\varepsilon$-hermitian form $\beta(\ovb)_0\in
  H_\varepsilon(N^\perp/N)$ that maps to $\beta(\ovb)\rvert_{N^\perp}$
  under the canonical map $H_\varepsilon(N^\perp/N)\to
  H_\varepsilon(N^\perp)$. Since
  $\beta(\ovb)\rvert_{N^\perp}=\beta(\ovb\rvert_{N^\perp})$, we have
  $\beta(\ovb)_0=\beta(\ovb_0)$. If $\ovb$ is nonsingular, then by
  definition $\beta(\ovb)$ is nonsingular. Then $\beta(\ovb)_0$ is
  nonsingular by Proposition~\ref{prop:sublagredherm}, hence $\ovb_0$
  is nonsingular.

  If $L\subset N^\perp/N$ is a totally isotropic submodule for
  $\ovb_0$, then the inverse image $L'\subset N^\perp$ of $L$ under
  the canonical map $N^\perp\to N^\perp/N$ is totally isotropic for
  $\ovb$. Therefore, $\ovb_0$ is anisotropic if $N$ is a maximal
  totally isotropic submodule.
\end{proof}

When $N$ is a maximal totally isotropic submodule of $M$, the quadratic
form $\ovb_0$ is called an \emph{anisotropic kernel} of $\ovb$. The
following result shows that, up to isometry, the anisotropic kernel
does not depend on the choice of the maximal totally isotropic
submodule: 

\begin{prop}
  \label{prop:uniquanisotker}
  All the maximal totally isotropic submodules of $M$ (for a given
  quadratic form $\ovb$) are isomorphic. If the form is nonsingular,
  then for any two isomorphic totally isotropic submodules $N$,
  $N'\subset M$ there is an isometry $\varphi$ of $(M,\ovb)$ such
  that $\varphi(N)=N'$.
\end{prop}

\begin{proof}
  See Tits \cite[Prop.~1 and 2]{Tits68}.
\end{proof}

\section{Quadratic forms on $A$-module bundles over a conic}
\label{sec:qfbndl}

Throughout this section, $C$ is a smooth projective conic over an
arbitrary field $F$, which we view as the Severi--Brauer variety of a
quaternion $F$-algebra $Q$. We assume $C$ has no rational point, which
amounts to saying that $Q$ is a division algebra.

\subsection{Vector bundles over $C$}
\label{subsec:vecbndl}

We recall from Roberts~\cite[\S2]{R} or
Biswas--Nagaraj~\cite{BN}\footnote{We are grateful to Van Geel for
  pointing out this reference.} the
description of vector bundles 
over $C$. (See the appendix for an elementary approach to vector
bundles over $C$.) Let $K$ be a separable quadratic extension of $F$ that
splits $Q$. Let $C_K=C\times\Spec K$ be the conic over
$K$ obtained by base change, and let $f\colon C_K\to C$ be the
projection. Since $C_K$ has a rational point, we have $C_K\simeq
\Proj$. By a theorem of Grothendieck, every vector bundle on $C_K$ is
a direct sum of vector bundles $\VR_{\Proj}(n)$ of rank~$1$ (see
Theorem~\ref{thm:Gro}). The vector 
bundle $f_*\bigl(\VR_{\Proj}(n)\bigr)$ is isomorphic to
$\VR_C(n)\oplus\VR_C(n)$ if $n$ is even; it is an indecomposable
vector bundle of rank~$2$ and degree~$2n$ if $n$ is odd
\cite[Theorem~1]{R} (see Corollary~\ref{cor:fO}). Letting
\[
\Indec_C(2n)=f_*\bigl(\VR_{\Proj}(n)\bigr)\qquad\text{for $n$ odd},
\]
it follows that every vector bundle over $C$ decomposes in a unique
way (up to isomorphism) as a direct sum of vector bundles of the type
$\VR_C(n)$ with $n$ even and $\Indec_C(2n)$ with $n$ odd (see
Theorem~\ref{thm:classbnd} or \cite[Theorem~4.1]{BN}). Moreover, we have
\begin{equation}
  \label{eq:EndIndec}
  \End\bigl(\Indec_C(2n)\bigr)\simeq Q \qquad\text{for all odd $n$.}
\end{equation}
(See \eqref{eq:EndIndec2}.) Using the property that $f_*\circ
f^*(\E)\simeq \E\oplus\E$ for every 
vector bundle $\E$ over $C$, and that $f^*\circ
f_*(\E')\simeq\E'\oplus\E'$ for every vector bundle $\E'$ over
$\Proj$ (see Proposition~\ref{prop:ff}), it is easy to see that
\begin{align}
  \label{eq:IndecIndec}
  &\Indec_C(2n)\otimes\Indec_C(2m) \simeq \VR_C(n+m)^{\oplus4}
     \qquad\text{for all odd $n$, $m$, and}\\
  \label{eq:IndecO}
  &\Indec_C(2n)\otimes\VR_C(m) \simeq \Indec_C\bigl(2(n+m)\bigr)  
     \qquad\text{for all $n$ odd and $m$ even.}
\end{align}

For each vector bundle $\E$ over $C$ we write $\E\spcheck =
\mathcal{H\!o\!m}(\E,\VR_C)$ for the dual vector bundle. Since for $n$
even $\VR_C(n)\spcheck$ is a vector bundle of rank~$1$ and
degree~$-n$, we have $\VR_C(n)\spcheck\simeq\VR_C(-n)$ for $n$
even. Similarly, $\Indec_C(2n)\spcheck\simeq\Indec_C(-2n)$ for $n$ odd
(see Corollary~\ref{cor:dual}).

\subsection{$A$-module bundles}

Let $A$ be a central simple algebra over $F$, and let $\E$ be a vector
bundle over $C$. A structure of \emph{right (resp.\ left) $A$-module
  bundle} on $\E$ is defined by a fixed $F$-algebra homomorphism
$A^{\operatorname{op}}\to\End\E$ (resp.\ $A\to\End\E$). Morphisms of
$A$-module bundles are morphisms of vector bundles that preserve the
action of $A$, hence for each $A$-module bundle $\E$ the $F$-algebra
$\End_A\E$ of $A$-module bundle endomorphisms is a subalgebra of the
finite-dimensional $F$-algebra $\End\E$ of vector bundle
endomorphisms. Therefore $\dim_F\End_A\E$ is finite, and by the same
argument as for vector bundles we have a Krull--Schmidt theorem for
$A$-module bundles: every $A$-module bundle over $C$ decomposes into a
direct sum of indecomposable $A$-module bundles, and this
decomposition is unique up to isomorphism. In this subsection, we
obtain information on the indecomposable $A$-module bundles. We
discuss only right $A$-module bundles; the case of left $A$-module
bundles is similar.

For every vector bundle $\E$ over $C$ and every right $A$-module $M$
of finite type, the tensor product over $F$ yields a right $A$-module
bundle $\E\otimes_FM$ with
\begin{equation}
\label{eq:Endotimes}
\End_A(\E\otimes_FM)=(\End\E)\otimes_F(\End_AM).
\end{equation}

\begin{prop}
  \label{prop:Amodirsum}
  Let $\E$ be a right $A$-module bundle over $C$, and let
  $\E^\natural$ be the vector bundle over $C$ obtained from $\E$ by
  forgetting the $A$-module structure. Then
  $\E$ is a direct summand of $\E^\natural\otimes_F A$.
\end{prop}

\begin{proof}
  Recall from \cite[(3.5)]{KMRT} that $A\otimes_FA$ contains a ``Goldman
  element'' $g=\sum a_i\otimes b_i$ characterized by the following
  property, where $\Trd_A$ denotes the reduced trace of $A$:
  \[
  \sum a_ixb_i=\Trd_A(x)\qquad\text{for all $x\in A$.}
  \]
  The element $g$ satisfies $(a\otimes 1)\cdot g = g\cdot(1\otimes a)$
  for all $a\in A$; see \cite[(3.6)]{KMRT}. Let $u\in A$ be such that
  $\Trd_A(u)=1$, hence $\sum a_iub_i=1$. Since $u\otimes1$ commutes with
  $1\otimes a$ for all $a\in A$, the element 
  \[
  g'=g\cdot(u\otimes1)=\sum a_iu\otimes b_i
  \]
  also satisfies $(a\otimes1)\cdot g'=g'\cdot(1\otimes a)$, hence
  \begin{equation}
  \label{eq:Goldmagic}
  \sum aa_iu\otimes b_i= \sum a_iu\otimes b_ia \qquad\text{for all
  $a\in A$.}
  \end{equation}
  Let $R$ be an arbitrary commutative $F$-algebra, and let
  $Q$ be a right $R\otimes_FA$-module. Let also $Q^\natural$ be the
  $R$-module obtained from $Q$ by forgetting the $A$-module structure.
  Because of \eqref{eq:Goldmagic},
  the map $Q\to Q^\natural\otimes_FA$ defined by $x\mapsto
  \sum(xa_iu)\otimes b_i$ is an $R\otimes_FA$-module
  homomorphism. Since $\sum a_iub_i=1$, this homomorphism is injective
  and split by the multiplication map $Q^\natural\otimes_F A\to
  Q$. This applies in particular to the module of sections of $\E$
  over any affine open set in $C$ and to the stalk of $\E$ at any
  point of $C$, and shows that $\E$ is a direct summand of
  $\E^\natural\otimes_F A$.
\end{proof}

\begin{cor}
  If $\E$ is an indecomposable $A$-module bundle, then all the
  indecomposable vector bundle summands in $\E^\natural$ are
  isomorphic. 
\end{cor}

\begin{proof}
  Let $\E^\natural=\Indec_1\oplus\cdots\oplus\Indec_r$ be the
  decomposition of $\E^\natural$ into indecomposable vector
  bundles. Then $\E^\natural\otimes A=(\Indec_1\otimes A)\oplus
  \cdots \oplus(\Indec_r\otimes A)$ is a decomposition of
  $\E^\natural\otimes A$ into $A$-module bundles. Since $\E$ is an
  indecomposable direct summand of $\E^\natural\otimes A$, it must be
  isomorphic to a direct summand of one of the $\Indec_i\otimes
  A$. But $(\Indec_i\otimes A)^\natural \simeq \Indec_i^{\oplus d}$,
  where $d=\dim A$, hence $\E^\natural\simeq \Indec_i^{\oplus m}$ for
  some $m$.
\end{proof}

If all the indecomposable direct summands in $\E^\natural$ are
isomorphic to $\Indec$, we say the indecomposable $A$-module bundle
$\E$ is \emph{of type $\Indec$}. Given the classification of
indecomposable vector bundles over $C$ in \S\ref{subsec:vecbndl}, we
may consider indecomposable $A$-module bundles of type $\VR_C(n)$ for
all even~$n$, and of type $\Indec_C(2n)$ for all odd~$n$. They are the
indecomposable $A$-module bundles in the decomposition of
$\VR_C(n)\otimes_FA$ and $\Indec_C(2n)\otimes_FA$ respectively. Since
$A$ is a direct sum of simple $A$-modules, they also are the
indecomposable summands in $\VR_C(n)\otimes_FM$ and
$\Indec_C(2n)\otimes_FM$ for any simple $A$-module $M$.

\begin{prop}
  \label{prop:clasAbndl}
  Let $M$ be a simple $A$-module.
  \begin{enumerate}
  \item[(i)]
  For $n$ even, $\VR_C(n)\otimes_FM$ is the unique indecomposable
  $A$-module bundle of type $\VR_C(n)$ up to isomorphism.
  \item[(ii)]
  For $n$ odd, there is a unique indecomposable $A$-module bundle
  $\E$ of type $\Indec_C(2n)$ up to isomorphism. This
  $A$-module bundle satisfies
  \[
  \Indec_C(2n)\otimes_FM \simeq \E^{\oplus\ell}
  \qquad\text{where $\ell=\frac{2\ind(A)}{\ind(Q\otimes_FA)}$.}
  \]
  \end{enumerate}
\end{prop}

Note that $\ind(Q\otimes_FA)$ may take the value $2\ind(A)$, $\ind(A)$
or $\frac12\ind(A)$, hence $\ell=1$, $2$ or $4$.

\begin{proof}
  (i)
  By~\eqref{eq:Endotimes} we have
  \[
  \End_A(\VR_C(n)\otimes_FM)=\bigl(\End\VR_C(n)\bigr)\otimes_F
  (\End_AM) = \End_AM.
  \]
  Since $M$ is simple, $\End_AM$ is a division algebra, hence
  $\VR_C(n)\otimes_FM$ is indecomposable.

  (ii)
  By~\eqref{eq:Endotimes} and \eqref{eq:EndIndec} we have
  \[
  \End_A(\Indec_C(2n)\otimes_FM) = \bigl(\End\Indec_C(2n)\bigr)
  \otimes_F(\End_AM) \simeq Q\otimes_F(\End_AM).
  \]
  This algebra is simple; it is isomorphic to $M_\ell(D)$ for $D$ a
  division algebra, hence $\Indec_C(2n)\otimes_FM$ decomposes into a
  direct sum of $\ell$ isomorphic $A$-module bundles.
\end{proof}

\subsection{Quadratic and Hermitian forms}

We keep the same notation as in the preceding subsections, and assume
$A$ carries an $F$-linear involution $\sigma$ (i.e., an involution of
the first kind). For every right $A$-module bundle $\E$ over $C$, we
define the \emph{dual bundle}
\[
\E^*=\mathcal{H\!o\!m}_{\VR_C\otimes A}(\E,\VR_C\otimes_FA).
\]
The bundle $\E^*$ has a natural structure of left $A$-module
bundle. Twisting the action of $A$ by $\sigma$, we may also consider
the right $A$-module bundle ${}^\sigma\E^*$, and define the vector
bundle
\[
\B(\E)={}^\sigma\E^*\otimes_A\E^*.
\]
As in \S\ref{sec:qf}, there is a switch map
$\sw\colon\B(\E)\to\B(\E)$. The kernel and cokernel of
$\Id\pm\sw$ define vector bundles over $C$. For $\delta=\pm1$, we
let
\[
\Herm_\delta(\E)=\ker(\Id-\delta\sw).
\]
Letting $\varepsilon=1$ if $\sigma$ is orthogonal and
$\varepsilon=-1$ if $\sigma$ is symplectic, we also define
\[
\Q(\E)=\coker(\Id-\varepsilon\sw).
\]

\begin{definition}
  A \emph{sesquilinear form} on the right $A$-module bundle $\E$ is a
  global section of $\B(\E)$. Likewise, a
  \emph{$\delta$-hermitian form} (resp.\ a \emph{quadratic form})
  on $\E$ is a global section of $\Herm_\delta(\E)$ (resp.\
  $\Q(\E)$). We write
  \[
  B(\E)=\Gamma\bigl(\B(\E)\bigr),\quad
  H_\delta(\E)=\Gamma\bigl(\Herm_\delta(\E)\bigr),\quad
  Q(\E)=\Gamma\bigl(\Q(\E)\bigr)
  \]
  for the $F$-vector spaces of sesquilinear, $\varepsilon$-hermitian,
  and quadratic forms respectively.
\end{definition}

\begin{prop}
  \label{prop:qfbndl}
  \begin{enumerate}
  \item[(i)]
  If $\E$ is an indecomposable $A$-module bundle of type $\VR_C(n)$
  with $n$ even, $n>0$, or of type $\Indec_C(2n)$ with $n$ odd, $n>0$,
  then for $\delta=\pm1$
  \[
  B(\E)=H_\delta(\E)=Q(\E)=\{0\}.
  \]
  \item[(ii)]
  If $\E=\VR_C(0)\otimes_FM$ for some right $A$-module $M$, then for
  $\delta=\pm1$
  \[
  B(\E)=B(M),\qquad H_\delta(\E)=H_\delta(M),\qquad Q(\E)=Q(M).
  \]
  \end{enumerate}
\end{prop}

\begin{proof}
  (i) It suffices to prove $B(\E)=\{0\}$. If
  $\E\simeq\VR_C(n)\otimes_FM$ for some simple $A$-module $M$, then
  $\E^*\simeq\VR_C(n)\spcheck\otimes_FM^*$, hence
  \[
  \B(\E)\simeq \VR_C(n)\spcheck \otimes_F \VR_C(n)\spcheck\otimes_F
  {}^\sigma M^*\otimes_AM^*\simeq \VR_C(-2n)\otimes_FB(M).
  \]
  Since $\Gamma\bigl(\VR_C(-2n)\bigr)=\{0\}$ for $n>0$ (see
  \eqref{eq:dimGamma}), it follows that $B(\E)=\{0\}$.

  If $\E$ is of type $\Indec_C(2n)$ with $n$ odd, then by
  Proposition~\ref{prop:clasAbndl} we have
  \[
  \Indec_C(2n)\otimes_FM\simeq\E^{\oplus\ell} \qquad\text{with
    $\ell=1$, $2$ or $4$,}
  \]
  hence
  \[
  B(\Indec_C(2n)\otimes_FM)\simeq B(\E)^{\oplus\ell^2}.
  \]
  Therefore, it suffices to prove $B(\Indec_C(2n)\otimes_FM)=\{0\}$
  for $n$ odd, $n>0$. As in the previous case we have
  \begin{align*}
  \B(\Indec_C(2n)\otimes_FM) & \simeq \Indec_C(2n)\spcheck\otimes_F
  \Indec_C(2n)\spcheck\otimes_F{}\sigma M^*\otimes_AM^* \\
  & \simeq
  \Indec_C(-2n)\otimes_F\Indec_C(-2n)\otimes_FB(M).
  \end{align*}
  By \eqref{eq:IndecIndec} it follows that
  \[
  \B(\Indec_C(2n)\otimes_FM)\simeq\VR_C(-2n)^{\oplus4}\otimes_FB(M).
  \]
  Since $\Gamma\bigl(\VR_C(-2n)\bigr)=\{0\}$ for $n>0$ (see
  \eqref{eq:dimGamma}), case~(i) of the proposition is proved.

  (ii)
  For $\E=\VR_C(0)\otimes_FM$ we have
  \[
  \B(\E)=\VR_C(0)\spcheck\otimes\VR_C(0)\spcheck\otimes_F{}^\sigma
  M^*\otimes_A M^* = \VR_C(0)\otimes_FB(M).
  \]
  Since $\Gamma\bigl(\VR_C(0)\bigr)=F$, it follows that $B(\E)=B(M)$,
  hence also $H_\delta(\E)=H_\delta(M)$ and $Q(\E)=Q(M)$.
\end{proof}

The property in~(ii) is expressed by saying that sesquilinear,
hermitian, and quadratic forms on $\VR_C(0)\otimes M$ are \emph{extended
  from $A$}.

We define the \emph{degree} of an $A$-module bundle $\E$ as the degree
of the underlying vector bundle $\E^\natural$.

\begin{thm}
  \label{thm:qfbndl}
  Let $\E$ be a right $A$-module bundle with $\deg\E=0$. If
  $\E$ carries a hermitian or quadratic form that is anisotropic on
  the generic fiber then $\E=\VR_C(0)\otimes N$ for some right
  $A$-module $N$.
\end{thm}

\begin{proof}
  Consider the decomposition of $\E$ into a direct sum of
  indecomposable $A$-module bundles. If any of the direct summand is
  of type $\VR_C(n)$ or $\Indec_C(2n)$ with $n>0$, then
  Proposition~\ref{prop:qfbndl}(i) shows that the restriction of any
  hermitian or quadratic form on $\E$ to this summand must be
  $0$. Therefore, if $\E$ carries an anisotropic hermitian or
  quadratic form, then all the summands must be of type $\VR_C(n)$
  with $n\leq0$ or $\Indec_C(2n)$ with $n<0$. But the degree of the
  indecomposable $A$-module bundles of type $\VR_C(n)$ or
  $\Indec_C(2n)$ with $n<0$ is strictly negative. Since $\deg\E=0$,
  all the summands are of type $\VR_C(0)$, hence by
  Proposition~\ref{prop:clasAbndl}(i) they are isomorphic to
  $\VR_C(0)\otimes_FM$ for $M$ a simple right $A$-module. Therefore,
  \[
  \E\simeq (\VR_C(0)\otimes M_1) \oplus\cdots\oplus (\VR_C(0)\otimes
  M_n) = \VR_C(0)\otimes(M_1\oplus\cdots\oplus M_n).
  \qedhere
  \]
\end{proof}

\begin{cor}
  \label{cor:qfbndl}
  If a right $A$-module bundle $\E$ with $\deg\E=0$ carries
  an anisotropic hermitian or quadratic form, then this form is
  extended from~$A$.
\end{cor}

\begin{proof}
  This readily follows from Proposition~\ref{prop:qfbndl}(ii) and
  Theorem~\ref{thm:qfbndl}. 
\end{proof}

We complete this section by discussing one case where the condition
$\deg\E = 0$ is necessarily satisfied. 

As for modules (see \eqref{eq:canisoB}), each $\delta$-hermitian form
$h\in H_\delta(\E)$ on a right $A$-module bundle $\E$ yields a morphism of
$A$-module bundles
\[
\widehat h\colon\E\to{}^\sigma\E^*.
\]

\begin{definition}
  The hermitian form $h$ on $\E$ is said to be \emph{nonsingular} if
  the morphism $\widehat h$ is an isomorphism.
\end{definition}

\begin{prop}
  \label{prop:nondeg0}
  If a right $A$-module bundle $\E$ carries a nonsingular
  $\delta$-hermitian form, then $\deg \E=0$.
\end{prop}

\begin{proof}
  We claim that $\deg{}^\sigma\E^*=-\deg\E$; therefore $\deg\E=0$ when
  $\E\simeq{}^\sigma\E^*$. It suffices to prove the claim for $\E$ an
  indecomposable $A$-module bundle, or indeed by
  Proposition~\ref{prop:clasAbndl}, for $\E$ of the form
  $\VR_C(n)\otimes_FM$ with $n$ even or $\Indec_C(2n)\otimes_FM$ with
  $n$ odd. We have
  \[
  {}^\sigma(\VR_C(n)\otimes_FM)^*=\VR_C(n)\spcheck\otimes_F{}^\sigma
  M^* \simeq \VR_C(-n)\otimes_F{}^\sigma M^*
  \]
  and
  \[
  {}^\sigma(\Indec_C(2n)\otimes_FM)^* = \Indec_C(2n)\spcheck\otimes_F
  {}^\sigma M^*\simeq \Indec_C(-2n)\otimes_F{}^\sigma M^*.
  \]
  The claim follows.
\end{proof}

\section{Excellence}
\label{sec:excel}

We use the same notation as in the preceding sections, and let $L$
denote the function field of the smooth projective conic $C$ over the
arbitrary field $F$. In this section, we prove that $L$ is excellent
for quadratic forms and hermitian forms on right $A$-modules. 

\subsection{Hermitian forms}
\label{subsec:excelherm}

Let $\delta=\pm1$, and let $h$ be a $\delta$-hermitian
form on a finitely generated right $A$-module $M$. Extending scalars
to $L$, we obtain a central simple $L$-algebra $A_L=L\otimes_FA$, a
right $A_L$-module $M_L=L\otimes_FM$, and a 
$\delta$-hermitian form $h_L$ on $M_L$. Scalar extension also yields
the right $A$-module bundle $\M_C=\VR_C(0)\otimes_FM$ over $C$, with
the $\delta$-hermitian form $h_C$ extended from~$h$.

For any $A_L$-submodule $N\subset M_L$, we let $\N$ denote the
intersection of the constant sheaf $N$ on $C$ with $\M_C$. This is a
vector bundle with stack
\[
\N_P=N\cap(\VR_P\otimes_FM)\qquad\text{at each point $P$ of $C$.}
\]
Following the elementary approach to vector bundles developed in the
appendix, the $A$-module bundle $\N$ is defined as follows: choose a
closed point $\infty=\Spec K$ on $C$ for some separable quadratic
extension $K$ of $F$, let $U=C\setminus\{\infty\}$, and define
$\N=(N,N_U,N_\infty)$ where
\[
N_U=N\cap(\VR_U\otimes_FM) \quad\text{and}\quad
N_\infty=N\cap(\VR_\infty\otimes_FM).
\]
The orthogonal of $N_U$ in $\VR_U\otimes_FM$ for the form extended
from $h$ is $N^\perp\cap(\VR_U\otimes_FM)$, and likewise the
orthogonal of $N_\infty$ in $\VR_\infty\otimes_FM$ is
$N^\perp\cap(\VR_\infty\otimes_FM)$, hence the orthogonal
$\N^\perp$ of $\N$ in $M_C$ is the $A$-module bundle
\[
\N^\perp=\bigl(N^\perp, N^\perp\cap(\VR_U\otimes_FM),
N^\perp\cap(\VR_\infty\otimes_FM)\bigr). 
\]
From here on, we assume $N\subset N^\perp$, hence $\N\subset\N^\perp$
and we may consider the quotient $A$-module bundle $\N^\perp/\N$. It
carries a $\delta$-hermitian form $h_0$ obtained by sublagrangian
reduction, see Proposition~\ref{prop:sublagredherm}.

For the excellence proof, the following result is key:

\begin{prop}
  \label{prop:nonsing}
  If $h$ is nonsingular, then the form $h_0$ on $\N^\perp/\N$ is
  nonsingular. 
\end{prop}

The proof uses the following lemma:

\begin{lemma}
  \label{lem:proj}
  Let $R$ be an $F$-algebra that is a Dedekind ring. Every finitely
  generated right $(R\otimes_FA)$-module that is torsion-free as an
  $R$-module is projective.
\end{lemma}

\begin{proof}
  Let $Q$ be a finitely generated right $(R\otimes_FA)$-module, and
  let $Q^\natural$ be the $R$-module obtained from $Q$ by forgetting
  the $A$-module structure. Recall from the proof of
  Proposition~\ref{prop:Amodirsum} that $Q$ is a direct summand of
  $Q^\natural\otimes_FA$. The $R$-module $Q^\natural$ is projective
  because it is finitely generated and torsion-free, hence
  $Q^\natural\otimes_FA$ is a projective $(R\otimes_FA)$-module. The
  lemma follows.
\end{proof}

\begin{proof}[Proof of Proposition~\ref{prop:nonsing}]
  Assume $h$ is nonsingular. Proposition~\ref{prop:sublagredherm}
  shows that the form $h_0$ is nonsingular on the generic fiber
  $N^\perp/N$ of $\N^\perp/\N$. We show that it is nonsingular on the
  stalk at each closed point of $C$.

  Fix some closed point $P$ of $C$, and let $\M_P=\VR_P\otimes_FM$ and
  $A_P=\VR_P\otimes_FA$. The right $A_P$-module $\M_P/\N_P$ is
  finitely generated and torsion-free as an $\VR_P$-module, hence it
  is projective by Lemma~\ref{lem:proj}, and the following exact
  sequence splits:
  \[
  0\to\N_P\to \M_P\to \M_P/\N_P\to 0.
  \]
  Lemma~\ref{lem:proj} also applies to show $\N_P^\perp/\N_P$ and
  $\M_P/\N_P$ are projective $A_P$-modules. On the other hand, the map
  $\widehat h_P=\Id\otimes\widehat h\colon\M_P\to{}^\sigma\M_P^*$ is
  bijective because $h$ is nonsingular. Substituting $\M_P$ for $M$
  and $\N_P$ for $N$ in the proof of
  Proposition~\ref{prop:sublagredherm}, we see that the arguments in
  that proof establish that the induced map $\N_P^\perp/\N_P\to
  {}^\sigma(\N_P^\perp/\N_P)^*$ is bijective.
\end{proof}

The excellence of $L$ for hermitian forms readily follows:

\begin{thm}
  \label{thm:excelherm}
  Let $h$ be a nonsingular $\delta$-hermitian form $(\delta=\pm1)$ on
  a finitely generated right $A$-module. The anisotropic kernel of
  $h_L$ is extended from $A$.
\end{thm}

\begin{proof}
  We apply the discussion above with $N\subset M_L$ a maximal
  sublagrangian. The induced $\delta$-hermitian form $h_0$ on
  $N^\perp/N$ is anisotropic by Proposition~\ref{prop:sublagredherm},
  and it is the generic fiber of a nonsingular $\delta$-hermitian form
  on the $A$-module bundle
  $\N^\perp/\N$ by
  Proposition~\ref{prop:nonsing}. Proposition~\ref{prop:nondeg0}
  yields 
  $\deg(\N^\perp/\N)=0$, hence Corollary~\ref{cor:qfbndl} shows that
  $h_0$ is extended from $A$.
\end{proof}

\subsection{Quadratic forms}
\label{subsec:excelquad}

We use the same notation as in \S\ref{subsec:excelherm}: $M$ is a
finitely generated right $A$-module and $\M_C=\VR_C(0)\otimes_FM$ is
the right $A$-module bundle obtained from $M$ by scalar extension,
with generic fiber $M_L$. We now consider a nonsingular quadratic form
$\ovb$ on 
$M$, and the extended quadratic form $\ovb_C$ on $\M_C$, with generic
fiber $\ovb_L$. Let $N\subset M_L$ be a maximal totally isotropic
subspace for $\ovb_L$. This subspace is totally isotropic (but maybe
not a maximal sublagrangian) for the hermitianized form
$\beta(\ovb_L)$, hence it lies 
in its orthogonal $N^\perp$ for $\beta(\ovb_L)$. By
Proposition~\ref{prop:sublagredquad}, $\ovb_L$ induces a nonsingular
quadratic 
form $\ovb_0$ on $N^\perp/N$, which is the anisotropic kernel of
$\ovb_L$. To prove that $L$ is excellent, we need to show that
$\ovb_0$ is extended from $A$.

The proof follows the same pattern as for
Theorem~\ref{thm:excelherm}. We consider the $A$-module bundles $\N$,
$\N^\perp$, and $\N^\perp/\N$ as in \S\ref{subsec:excelherm}. As
observed in the proof of Proposition~\ref{prop:nonsing}, for each
closed point $P$ of $C$ the $A_P$-modules $\M_P/\N_P$, $\M_P/\N_P^\perp$,
and $\N_P^\perp/\N_P$ are projective. Substituting $\M_P$ for $M$ and
$\N_P$ for $N$ in the proof of Proposition~\ref{prop:sublagredquad},
we see that the form $q_0$ is the generic fiber of a nonsingular
quadratic form $q_0$ on $\N_P^\perp/\N_P$. We have
$\deg(\N^\perp/\N)=0$ by Proposition~\ref{prop:nondeg0}, and since
$\ovb_0$ is anisotropic on $N^\perp/N$ it is extended from~$A$ by
Corollary~\ref{cor:qfbndl}. We have thus proved:

\begin{thm}
  \label{thm:excelquad}
  Let $\ovb$ be a nonsingular quadratic form on a finitely generated
  right $A$-module. The anisotropic kernel of $\ovb_L$ is extended
  from~$A$. 
\end{thm}

\section*{Appendix: Vector bundles over conics}
\renewcommand{\thesection}{\Alph{section}}
\setcounter{section}{1}
\setcounter{subsection}{0}
\setcounter{lemma}{0}

We give in this appendix an elementary proof of the classification of
vector bundles over conics used in \S\ref{sec:qfbndl}. The elementary
character of our approach is based on the representation of vector
bundles over conics or over the projective line as triples consisting
of the generic fiber, the module of sections over an affine open set,
and the stalks at the complement, which consists in one or two closed
points; see \S\ref{sec:VBP} and \S\ref{sec:VBc}.

\subsection{Matrices}
\label{sec:mat}

Let $K$ be an arbitrary field and let $u$ be an indeterminate on
$K$. Let $w_0$ and $w_\infty$ be respectively the $u$-adic and the
$u^{-1}$-adic valuations on the field $K(u)$ (with value group
$\mathbb{Z}$). Consider the following subrings of $K(u)$:
\[
\VR_V=K[u,u^{-1}],\qquad \VR_S=\{x\in K(u)\mid w_0(x)\geq0 \text{ and
} w_\infty(x)\geq0\}.
\]

The following theorem is equivalent to Grothendieck's classification
of vector bundles over the projective line~\cite{Groth}, as we will
see in \S\ref{sec:VBP}. (See~\cite{HM} for an 
elementary proof of another statement on matrices that is equivalent
to Grothendieck's theorem.)

\begin{thm}
  \label{thm:Grot}
  For every matrix $g\in\GL_n(K(u))$ there exist matrices
  $p\in\GL_n(\VR_S)$ and $q\in\GL_n(\VR_V)$  such that
  \[
  pgq=\diag\bigl((u-1)^{k_1},\ldots,(u-1)^{k_n}\bigr)\qquad
  \text{for some $k_1$, \ldots, $k_n\in\mathbb{Z}$}.
  \]
\end{thm}

\begin{proof}
  The case $n=1$ is easy: using unique factorization in $K[u]$, we may
  factor every element in $K(u)^\times$ as $g=p\cdot(u-1)^k\cdot
  u^\alpha$ where $w_0(p)=w_\infty(p)=0$, hence
  $p\in\VR_S^\times$. The rest of the proof is by induction on $n$. In
  view of the $n=1$ case, it suffices to show that we may find
  $p\in\GL_n(\VR_S)$, $q\in\GL_n(\VR_V)$ such that $p\cdot g\cdot q$
  is diagonal. Since $\VR_V$ is a principal ideal domain, we may find
  a matrix $q_1\in\GL_n(\VR_V)$ such that
  \[
  gq_1=
  \begin{pmatrix}
    a_1&0&\cdots&0\\
    * & & & \\
    \vdots& &g_1 & \\
    * & & & 
  \end{pmatrix}
  \]
  where $a_1$ is the gcd of the entries in the first row of $g$. By
  induction, we may assume the theorem holds for $g_1$ and thus find
  $p_2\in\GL_n(\VR_S)$, $q_2\in\GL_n(\VR_V)$ such that
  \[
  p_2gq_1q_2=
  \begin{pmatrix}
    a_1 & 0 & 0 &\cdots&0\\
    b_2 & a_2&0&\cdots&0\\
    b_3 & 0& a_3&\cdots&0\\
    \vdots&\vdots&\vdots&\ddots&\vdots\\
    b_n &0&0&\cdots&a_n
  \end{pmatrix}
  \]
  for some $a_2$, \ldots, $a_n\in K(u)^\times$ and some $b_2$, \ldots,
  $b_n\in K(u)$. To complete the proof, it now suffices to apply
  $(n-1)$ times the following lemma:
  \renewcommand{\qed}{\relax}
\end{proof}

\begin{lemma}
  \label{lem:diag}
  Let $a$, $b$, $c\in K(u)$ with $a$, $c\neq0$. There exists
  $p\in\GL_2(\VR_S)$, $q\in\GL_2(\VR_V)$ such that the matrix
  \[
  p\cdot
  \begin{pmatrix}
    a&0\\ b&c
  \end{pmatrix}\cdot q
  \]
  is diagonal.
\end{lemma}

The proof uses the following approximation property:

\begin{prop}
  \label{prop:approx}
  For every $f\in K(u)^\times$, there exists $\lambda\in\VR_V$ such
  that $w_0(f-\lambda)\geq0$ and $w_\infty(f-\lambda)>0$.
\end{prop}

\begin{proof}
  We first show, by descending induction on $w_0(f)$, that there
  exists $\lambda_0\in\VR_V$ such that $w_0(f-\lambda_0)\geq0$: if
  $w_0(f)\geq0$ we may take $\lambda_0=0$. Otherwise, let
  $f=ab^{-1}u^\alpha$ where $a$, $b\in F[u]$ are not divisible by
  $u$. For $\mu=a(0)b(0)^{-1}u^\alpha\in\VR_V$ we have
  \[
  w_0(f-\mu)>\alpha=w_0(f),
  \]
  hence induction yields $\mu_0\in\VR_V$ such that
  $w_0\bigl((f-\mu)-\mu_0\bigr)\geq0$, and we may take
  $\lambda_0=\mu+\mu_0$.

  Fix $\lambda_0\in\VR_V$ such that $w_0(f-\lambda_0)\geq0$. If
  $w_\infty(f-\lambda_0)>0$ we are done. Otherwise, let
  \[
  f-\lambda_0=\frac{a_nu^n+\cdots+a_0}{b_mu^m+\cdots+b_0}
  \]
  with $a_n$, \ldots, $a_0$, $b_m$, \ldots, $b_0\in K$, $a_n$,
  $b_m\neq0$, so that $w_\infty(f-\lambda_0)=m-n\leq0$. Let
  $\mu_1=a_nb_m^{-1}u^{n-m}\in F[u]$. We have
  \[
  w_\infty\bigl((f-\lambda_0)-\mu_1\bigr)>m-n=w_\infty(f-\lambda_0).
  \]
  Again, arguing by induction on $w_\infty(f-\lambda_0)$, we may find
  $\mu_2\in F[u]$ such that
  \[
  w_\infty\bigl((f-\lambda_0)-\mu_2\bigr)>0.
  \]
  Note that $w_0(\mu_2)\geq0$ since $\mu_2\in F[u]$. Therefore, 
  \[
  w_0\bigl((f-\lambda_0)-\mu_2\bigr)\geq\min\bigl(w_0(f-\lambda_0),
  w_0(\mu_2)\bigr)\geq0,
  \]
  so we may choose $\lambda=\lambda_0+\mu_2$.
\end{proof}

\begin{proof}[Proof of Lemma~\ref{lem:diag}]
  For $f\in K(u)^\times$, let $w(f)=w_0(f)+w_\infty(f)$. Note that $w$
  is \emph{not} a valuation, but it is multiplicative and $w(u)=0$. We
  shall argue by induction on $w(a)-w(c)\in\mathbb{Z}$; but first note
  that by multiplying $\bigl(
  \begin{smallmatrix}
    a&0\\b&c
  \end{smallmatrix}\bigr)$
  on the right by $\bigl(
  \begin{smallmatrix}
    1&0\\0&u^\alpha
  \end{smallmatrix}\bigr)$
  for $\alpha=w_0(a)-w_0(c)$, we may assume $w_0(a)=w_0(c)$. By
  Proposition~\ref{prop:approx}, there exists $\lambda\in\VR_V$ such
  that
  \[
  w_0(bc^{-1}-\lambda)\geq0\qquad\text{and}\qquad
  w_\infty(bc^{-1}-\lambda)>0.
  \]
  We then have $w_0(b-\lambda c)\geq w_0(c)=w_0(a)$ and
  $w_\infty(b-\lambda c)> w_\infty(c)$. Multiplying $\bigl(
  \begin{smallmatrix}
    a&0\\b&c
  \end{smallmatrix}\bigr)$ on the right by $\bigl(
  \begin{smallmatrix}
    1&0\\-\lambda&1
  \end{smallmatrix}\bigr)$ yields
  \[
  \begin{pmatrix}
    a&0\\b&c
  \end{pmatrix}\cdot
  \begin{pmatrix}
    1&0\\-\lambda&1
  \end{pmatrix}=
  \begin{pmatrix}
    a&0\\b-\lambda c&c
  \end{pmatrix}.
  \]
  Thus, we may substitute $b-\lambda c$ for $b$ and thus assume
  \begin{equation}
    \label{eq:wb}
    w_0(b)\geq w_0(c)=w_0(a) \qquad\text{and}\qquad w_\infty(b)>
    w_\infty(c).
  \end{equation}
  If $w_\infty(b)\geq w_\infty(a)$, then $a^{-1}b\in\VR_S$ and the
  lemma follows from the equation
  \begin{equation}
    \label{eq:mat}
    \begin{pmatrix}
      1&0\\-a^{-1}b&1
    \end{pmatrix}\cdot
    \begin{pmatrix}
      a&0\\b&c
    \end{pmatrix}=
    \begin{pmatrix}
      a&0\\0&c
    \end{pmatrix}.
  \end{equation}
  We now start our induction on $w(a)-w(c)$. If $w(a)-w(c)\leq0$, then
  since $w_0(a)=w_0(c)$ we have $w_\infty(a)\leq w_\infty(c)$. By
  \eqref{eq:wb} it follows that $w_\infty(b)>w_\infty(a)$ and we are
  done by \eqref{eq:mat}. If $w(a)-w(c)>0$ but $w_\infty(b)\geq
  w_\infty(a)$, we may also conclude by \eqref{eq:mat}. For the rest
  of the proof, we may thus assume
  $w_\infty(a)>w_\infty(b)>w_\infty(c)$. If $w_0(b)>w_0(a)$, then in
  view of the equation
  \[
  \begin{pmatrix}
    1&0\\1&1
  \end{pmatrix}\cdot
  \begin{pmatrix}
    a&0\\b&c
  \end{pmatrix}=
  \begin{pmatrix}
    a&0\\a+b&c
  \end{pmatrix}
  \]
  we may substitute $a+b$ for $b$. In that case, we have
  \[
  w_0(a+b)=\min\bigl(w_0(a),w_0(b)\bigr)=w_0(a)
  \]
  and
  \[
  w_\infty(a+b)=\min\bigl(w_\infty(a),w_\infty(b)\bigr)=w_\infty(b).
  \]
  Thus, in all cases we may assume
  \[
  w_0(b)=w_0(a)=w_0(c)\qquad\text{and}\qquad
  w_\infty(a)>w_\infty(b)>w_\infty(c).
  \]
  Then $ab^{-1}\in\VR_S$. Consider
  \[
  \begin{pmatrix}
    1&-ab^{-1}\\0&1
  \end{pmatrix}\cdot
  \begin{pmatrix}
    a&0\\b&c
  \end{pmatrix}\cdot
  \begin{pmatrix}
    0&1\\1&0
  \end{pmatrix}=
  \begin{pmatrix}
    -ab^{-1}c&0\\c&b
  \end{pmatrix}.
  \]
  We have 
  \[
  w(-ab^{-1}c)-w(b)=w(a)+w(c)-2w(b)=w_\infty(a)+w_\infty(c)-2w_\infty(b).
  \]
  Since $w_\infty(b)>w_\infty(c)$ we have
  \[
  w_\infty(a)+w_\infty(c)-2w_\infty(b)< w_\infty(a)-w_\infty(c).
  \]
  But $w(a)-w(c)=w_\infty(a)-w_\infty(c)$, hence
  $w(-ab^{-1}c)-w(b)<w(a)-w(c)$. By induction, the lemma holds for
  $\bigl(
  \begin{smallmatrix}
    -ab^{-1}c&0\\c&b
  \end{smallmatrix}\bigr)$, hence also for
  $\bigl(
  \begin{smallmatrix}
    a&0\\b&c
  \end{smallmatrix}\bigr)$.
\end{proof}

\subsection{Vector bundles over $\Proj$}
\label{sec:VBP}

We use the same notation as in \S\ref{sec:mat}.

\begin{defis}
  A \emph{vector bundle} over $\Proj$ is a triple $\E=(E,E_V,E_S)$
  consisting of a finite-dimensional $K(u)$-vector space $E$, a
  finitely generated $\VR_V$-module $E_V\subset E$, and a finitely
  generated $\VR_S$-module $E_S\subset E$ such that
  \[
  E=E_V\otimes_{\VR_V}K(u) = E_S\otimes_{\VR_S}K(u).
  \]
  The \emph{rank} of $\E$ is $\rk \E=\dim E$. The intersection $E_V\cap
  E_S$ is a $K$-vector space, which is called the space of
  \emph{global sections} of $\E$. We use the notation
  \[
  \Gamma(\E)=E_V\cap E_S.
  \]
  Since $\VR_V$ and $\VR_S$ are principal ideal
  domains, the $\VR_V$- and $\VR_S$-modules $E_V$ and $E_S$ are
  free. Their rank is the rank $n$ of $\E$. Let $(e_i)_{i=1}^n$ (resp.\
  $(f_i)_{i=1}^n$) be a base of the $\VR_V$-module $E_V$ (resp.\ the
  $\VR_S$-module $E_S$). Each of these bases is a $K(u)$-base of $E$,
  hence we may find a matrix $g=(g_{ij})_{i,j=1}^n\in\GL_n(K(u))$ such
  that
  \[
    e_j=\sum_{i=1}^nf_ig_{ij}\qquad\text{for $j=1$, \ldots, $n$.}
  \]
  The \emph{degree} $\deg \E$ is defined as
  \[
  \deg \E = w_0(\det g) + w_\infty(\det g)\in\mathbb{Z}.
  \]
  To see that this integer does not depend on the choice of bases,
  observe that a change of bases substitutes for the matrix $g$ a
  matrix $g'$ of the form $g'=pgq$ for some $p\in\GL_n(\VR_S)$ and
  $q\in\GL_n(\VR_V)$. We have $\det p\in\VR_S^\times$, hence $w_0(\det
  p)=w_\infty(\det p)=0$. Likewise, $\det
  q\in\VR_V^\times=K^\times\oplus u^{\mathbb Z}$, so $w_0(\det
  q)+w_\infty(\det q) =0$, and it follows that $w_0(\det g) +
  w_\infty(\det g) = w_0(\det g')+w_\infty(\det g')$.

  A \emph{morphism} of vector bundles $(E,E_V,E_S)\to (E',E'_V,E'_S)$
  over $\Proj$ is a $K(u)$-linear map $\varphi\colon E\to E'$ such
  that $\varphi(E_V)\subset E'_V$ and $\varphi(E_S)\subset E'_S$.
\end{defis}

\begin{example}
  \label{ex:rk1}
  \emph{Vector bundles of rank~$1$.} 
  Since $\VR_V$ and $\VR_S$ are principal ideal domains, every vector
  bundle of rank~$1$ is isomorphic to a triple $\E=(K(u), f\VR_V,
  g\VR_S)$ for some $f$, $g\in K(u)^\times$. Using unique
  factorization in $K[u]$ we may find $p\in\VR_S^\times$, $k$,
  $\alpha\in\mathbb{Z}$ such that $fg^{-1}=p\cdot(u-1)^k\cdot
  u^\alpha$. Multiplication by $g^{-1}p^{-1}(u-1)^{-k}$ is a
  $K(u)$-linear map $\varphi\colon K(u)\to K(u)$ such that
  $\varphi(f)=u^\alpha$ and $\varphi(g)=p^{-1}(u-1)^{-k}$. Since
  $u\in\VR_V^\times$, it follows that
  $\varphi(f\VR_V)=\VR_V$. Likewise, since
  $p\in\VR_S^\times$, we have
  $\varphi(g\VR_S)=(u-1)^{-k}\VR_S$. Therefore, $\varphi$ defines an
  isomorphism $\E\stackrel{\sim}{\to} (K(u),\VR_V,
  (u-1)^{-k}\VR_S)$. For $n\in\mathbb{Z}$, we write
  \[
  \VR_{\Proj}(n)=(K(u),\VR_V,(u-1)^n\VR_S).
  \]
  If $g\in K(u)^\times$ satisfies $w_0(g)+w_\infty(g)= -n$, then
  $g\cdot(u-1)^{-n} u^{-w_0(g)}\in\VR_S^\times$, hence the arguments
  above yield
  \begin{equation}
    \label{eq:On}
    (K(u),\VR_V,g\VR_S) \simeq (K(u),\VR_V,(u-1)^n\VR_S) =
    \VR_{\Proj}(-w_0(g)-w_\infty(g)). 
  \end{equation}
  By definition of the degree,
  \[
  \deg\VR_{\Proj}(n)=w_0\bigl((u-1)^{-n}\bigr) +
  w_\infty\bigl((u-1)^{-n}\bigr) = n.
  \]
  The vector space of global sections of $\VR_{\Proj}(n)$ is easily
  determined: by definition, we have
  \begin{align*}
  \Gamma(\VR_{\Proj}(n)) & =\VR_V\cap (u-1)^n\VR_S \\ & = \{f\in\VR_V
  \mid w_0(f)\geq w_0((u-1)^n), \; w_\infty(f)\geq w_\infty((u-1)^n)\}.
  \end{align*}
  Since $w_0(u-1)=0$ and $w_\infty(u-1)=-1$, we have
  \[
  \Gamma(\VR_{\Proj}(n))=\{f\in K[u]\mid \deg f\leq n\},
  \]
  hence
  \[
  \dim\Gamma(\VR_{\Proj}(n))=
  \begin{cases}
    0&\text{if $n<0$},\\
   1+n&\text{if $n\geq0$}.
  \end{cases}
  \]
\end{example}

\begin{thm}[Grothendieck]
  \label{thm:Gro}
  For every vector bundle $\E$ on $\Proj$, there exist integers
  $k_1$, \ldots, $k_n\in\mathbb{Z}$ such that
  \[
  \E\simeq \VR_{\Proj}(k_1)\oplus\cdots\oplus\VR_{\Proj}(k_n).
  \]
\end{thm}

\begin{proof}
  Let $\E=(E,E_V,E_S)$ be of rank~$n$. Let $(e_i)_{i=1}^n$ (resp.\
  $(f_i)_{i=1}^n$) be a base of the $\VR_V$-module $E_V$ (resp.\ the
  $\VR_S$-module $E_S$), and let $g=(g_{ij})_{i,j=1}^n\in\GL_n(K(u))$ be such
  that
  \begin{equation}
    \label{eq:1}
    e_j=\sum_{i=1}^nf_ig_{ij}\qquad\text{for $j=1$, \ldots, $n$.}
  \end{equation}
  Theorem~\ref{thm:Grot} yields matrices $p\in\GL_n(\VR_S)$ and
  $q\in\GL_n(\VR_V)$ such that
  \begin{equation}
    \label{eq:2}
    pgq=\diag\bigl((u-1)^{-k_1},\ldots, (u-1)^{-k_n}\bigr)
    \qquad\text{for some $k_1$, \ldots, $k_n\in\mathbb{Z}$.}
  \end{equation}
  Let $p^{-1}=(p_{ij})_{i,j=1}^n$ and $q=(q_{ij})_{i,j=1}^n$, and
  define for $j=1$, \ldots, $n$
  \[
  f'_j=\sum_{i=1}^nf_ip_{ij} \quad\text{and}\quad
  e'_j=\sum_{i=1}^ne_iq_{ij}.
  \]
  Because $p\in\GL_n(\VR_S)$, the sequence $(f'_i)_{i=1}^n$ is a base
  of $E_S$. Likewise, $(e'_i)_{i=1}^n$ is a base of $E_V$, and
  from~\eqref{eq:1} and \eqref{eq:2} we derive $e'_i=f'_i(u-1)^{-k_i}$
  for $i=1$, \ldots, $n$. Thus,
  \[
  E=\bigoplus_{i=1}^ne'_iK(u),\quad E_V=\bigoplus_{i=1}^n e'_i\VR_V,
  \quad
  E_S = \bigoplus_{i=1}^ne'_i(u-1)^{k_i}\VR_S.
  \]
  These equations mean that the map $E\to K(u)^{\oplus n}$ that
  carries each vector to the $n$-tuple of its coordinates in the base
  $(e'_i)_{i=1}^n$ defines an isomorphism of vector bundles
  \[
  \E\stackrel{\sim}{\to}\VR_{\Proj}(k_1)\oplus\cdots\oplus\VR_{\Proj}(k_n).
  \qedhere
  \]
\end{proof}

\begin{cor}
  \label{cor:glsecP}
  For every vector bundle $\E$ on $\Proj$, the $K$-vector space of
  global sections $\Gamma(\E)$ is finite-dimensional. More precisely,
  if $\E\simeq\VR_{\Proj}(k_1)\oplus\cdots\oplus\VR_{\Proj}(k_n)$ for
  some $k_1$, 
  \ldots, $k_n\in\mathbb{Z}$, then
  \[
  \dim\Gamma(\E)=\sum_{i=1}^n\max(1+k_i,0) \quad\text{and}\quad \deg \E=
  \sum_{i=1}^n k_i.
  \]
\end{cor}

\begin{proof}
  If $\E=\E_1\oplus \E_2$, then $\Gamma(\E)=\Gamma(\E_1)\oplus\Gamma(\E_2)$
  and $\deg \E=\deg \E_1+\deg \E_2$. Since each $\Gamma(\VR_{\Proj}(n))$ is
  finite-dimensional and $\deg\VR_{\Proj}(n)=n$ (see
  Example~\ref{ex:rk1}), the corollary follows.
\end{proof}

From the formula for $\dim\Gamma(\E)$, it is easily seen by tensoring $\E$
with $\VR_{\Proj}(k)$ for various $k\in\mathbb{Z}$ that the integers $k_1$,
\ldots, $k_n$ such that
$\E\simeq\VR_{\Proj}(k_1)\oplus\cdots\oplus\VR_{\Proj}(k_n)$ 
are uniquely determined up to permutation.

\subsection{Vector bundles over conics}
\label{sec:VBc}

Let $L$ be the function field of a smooth projective conic $C$ over a field
$F$. Assume $C$ has no rational point over $F$, and let $\infty$ be a
point of degree~$2$ on $C$ with residue field $K$ separable over
$F$. Let $v_\infty$ be the corresponding discrete valuation on $L$ and
$\VR_\infty$ be its valuation ring. Let also $\VR_U\subset L$ be the
affine ring of $C\setminus\{\infty\}$, which is the intersection of
all the valuation rings of the $F$-valuations on $L$ other than
$v_\infty$.

Let $C_K=C\times\operatorname{Spec} K$ be the conic over $K$ obtained
by base change, and let $f\colon C_K\to C$ be the projection. Since
$C_K$ has a rational point, we have $C_K\simeq\Proj$, i.e., the
composite field $KL$ is a purely transcendental extension of $K$. We
may find $u\in KL$ such that $KL=K(u)$ and the two valuations of
$K(u)$ extending $v_\infty$ are $w_0$ and $w_\infty$, the $u$-adic and
$u^{-1}$-adic valuation of $K(u)$. Thus,
using the notation of \S\ref{sec:VBP}, 
\[
\VR_U\otimes_FK=\VR_V \quad\text{and}\quad \VR_\infty\otimes_FK=\VR_S.
\]

\begin{remark}
  \label{rem:explicit}
  A concrete description of the rings defined above can be obtained by
  representing $C$ as the Severi--Brauer variety of a quaternion
  division algebra $Q$. Write $V$ for the $3$-dimensional subspace of
  trace $0$ quaternions. Then $q(v):=v^2$ is a quadratic form on $V$
  and the conic $C$ is the quadric in the projective plane $\mathbb{P}(V)$
  given by the equation $q=0$. Every closed point of degree~$2$ on $C$
  is determined by an equation $\varphi=0$ for some nonzero linear form
  $\varphi\in V^*$. If $(r,s)$ is a base of $\ker\varphi\subset V$,
  then the equation $(xr+ys)^2=0$ has the solution $x=-q(s)$, $y=rs$
  in $F(rs)$, hence $F(rs)$ is the residue field of the corresponding
  point. Let $\infty$ be the closed point on $C$ determined by a
  linear form $\varphi$ such that $F(rs)$ is a separable quadratic
  extension of $F$. Let also $t\in V$ be a nonzero vector orthogonal
  to $\ker\varphi$ for the polar form $b_q$ of $q$. If
  $t\in\ker\varphi$, then $b_q(t,t)=0$, hence $\charac F=2$. Moreover,
  $t$ is a linear combination of $r$ and $s$, and the equations
  $b_q(t,r)=b_q(t,s)=0$ yield $b_q(r,s)=0$. This is a contradiction
  because then the minimal polynomial of $rs$, which is
  $X^2-b_q(r,s)X+q(r)q(s)$, is not separable. Therefore, in all cases
  the choice of $\infty$ guarantees that $(r,s,t)$ is a base of
  $V$. Let $(x,y,z)$ be the dual base of $V^*$. Then the conic $C$
  is given by the equation 
  \[
  (xr+ys+zt)^2=0,
  \]
  and $\infty$ is the point determined by the equation $z=0$. Because
  $t$ is orthogonal to $r$ and $s$, the equation of the conic
  simplifies to
  \[
  (xr+ys)^2+z^2t^2=0.
  \]
  Let $U=C\setminus\{\infty\}$; then
  \[
  \VR_U=F \Big [\frac{x}{z}, \frac{y}{z}\Big ]\subset F\Big
  (\frac{x}{z}, \frac{y}{z}\Big )=L.
  \] 
  The equation of the conic shows that $\frac yz$ is a root of a
  quadratic equation over $F(\frac xz)$, hence every element in $L$
  has a unique expression of the form $f(\frac xz)+\frac yz g(\frac
  xz)$ for some rational functions $f$, $g$ with coefficients in $F$.
  If $v_{\infty}$ is the discrete valuation of the local ring
  $\VR_{\infty}$, then  
  \[
  v_{\infty}\Big (\frac{x}{z}\Big )=v_{\infty}\Big (\frac{y}{z}\Big )=-1.
  \]
  More precisely, for $f$, $g$, $h$ polynomials in one variable over
  $F$, with $h\neq0$,
  \[
  v_\infty\Big(\frac{f(\frac xz)+\frac yzg(\frac xz)}{h(\frac
    xz)}\Big) = \deg h-\max(\deg f,\: 1+\deg g).
  \]
  We claim that we may take for $u$ the element $\frac xzrs+\frac
  yzq(s)$. To see this, let $\iota$ denote the nontrivial
  $L$-automorphism of $KL$.  For $u=\frac xzrs+\frac yzq(s)$ we have
  $\iota(u)=\frac xzsr+\frac yzq(s)$, and from the equation of the conic
  it follows that
  \begin{equation}
  \label{eq:iotau}
  u\:\iota(u)=\frac{q(s)}{z^2}(xr+ys)^2=-q(s)q(t)\in F^\times.
  \end{equation}
  This equation shows that for every valuation $w$ of $KL$ extending
  $v_\infty$ we have $w(u)=-w\bigl(\iota(u)\bigr)$. Moreover, from
  $u=\frac xzrs+\frac yzq(s)$ and $u-\iota(u)=\frac xz(rs-sr)$ it
  follows that
  \[
  w(u)\geq\min\Bigl(v_\infty\Bigl(\frac xz\Bigr),v_\infty\Bigl(\frac
  yz\Bigr)\Bigr)=-1 \quad\text{and}\quad -1=v_\infty\Bigl(\frac xz\Bigr)
  \geq\min\bigl(w(u),w\bigl(\iota(u)\bigr)\bigr).
  \]
  Therefore, either $w(u)=-w\bigl(\iota(u)\bigr)=1$, i.e., $w=w_0$, or
  $w(u)=w\bigl(\iota(u)\bigr)=-1$, i.e., $w=w_\infty$. 
\end{remark}

The following result is folklore. (For a proof in characteristic
different from~$2$, see Pfister \cite[Prop.~1]{P}.)

\begin{lemma}
  The ring $\VR_U$ is a principal ideal domain.
\end{lemma}

\begin{proof}
  Let $I\subset\VR_U$ be an ideal. Since $\VR_V=K[u,u^{-1}]$ is a
  principal ideal domain, we may find $f\in\VR_V$ such that
  $I\otimes_FK=f\VR_V$. As $I\otimes_FK$ is preserved by $\iota$, we
  have $f\VR_V=\iota(f)\VR_V$, hence
  $\iota(f)f^{-1}\in\VR_V^\times=K^\times\oplus u^{\mathbb Z}$. Let
  $a\in K^\times$ and $\alpha\in\mathbb Z$ be such that
  \begin{equation}
    \label{eq:iotaf}
    \iota(f)f^{-1}=au^\alpha.
  \end{equation}
  Since $N_{KL/L}(\iota(f)f^{-1})=1$, it follows by~\eqref{eq:iotau}
  that
  \[
  N_{KL/L}(au^\alpha)=N_{K/F}(a)\bigl(-q(s)q(t)\bigr)^\alpha=1.
  \]
  If $\alpha$ is odd, let $\alpha=2\beta-1$ and
  $a\bigl(-q(s)q(t)\bigr)^\beta=b+crs$ with $b$, $c\in F$. Then
  $N_{K/F}(b+crs)= -q(s)q(t)$, hence
  \[
  (cr+bq(s)^{-1}s)^2+t^2=0.
  \]
  Thus, the conic $C$ has an $F$-rational point, a
  contradiction. Therefore, $\alpha$ is even. Let
  $\alpha=2\beta$. Then from \eqref{eq:iotau} and \eqref{eq:iotaf} we
  have
  \[
  \iota(u^\beta f)\cdot(u^\beta
  f)^{-1}=a\bigl(-q(s)q(t)\bigr)^\beta\in K^\times.
  \]
  By Hilbert's Theorem~90, we may find $b\in K^\times$ such that
  $a\bigl(-q(s)q(t)\bigr)^\beta=b\iota(b)^{-1}$. Then
  \[
  \iota(bu^\beta f)=bu^\beta f\in L^\times.
  \]
  Since $bu^\beta\in\VR_V^\times$, we have $f\VR_V=bu^\beta f\VR_V$,
  hence $I=bu^\beta f\VR_U$.
\end{proof}

\begin{defis}
  A \emph{vector bundle} over $C$ is a triple $\E=(E,E_U,E_\infty)$
  consisting of a finite-dimensional $L$-vector space $E$, a finitely
  generated $\VR_U$-module $E_U\subset E$, and a finitely generated
  $\VR_\infty$-module $E_\infty\subset E$ such that
  \[
  E=E_U\otimes_{\VR_U}L=E_\infty\otimes_{\VR_\infty}L.
  \]
  The \emph{rank} of $\E$ is $\rk \E=\dim E$. The intersection
  $E_U\cap E_\infty$ is an $F$-vector space called the space of
  \emph{global sections} of $\E$. We write
  \[
  \Gamma(\E)=E_U\cap E_\infty.
  \]
  The degree of a vector bundle over $C$ is defined as for vector
  bundles over $\Proj$: Since $\VR_U$ and $\VR_\infty$ are
  principal ideal domains, the $\VR_U$- and $\VR_\infty$-modules $E_U$
  and $E_\infty$ are free of rank $\rk\E$. Let $(e_i)_{i=1}^n$ (resp.\
  $(f_i)_{i=1}^n$) be a base of the $\VR_U$-module $E_U$ (resp.\ the
  $\VR_\infty$-module $E_\infty$). Each of these bases is an $L$-base
  of $E$, hence we may find a matrix $g=(g_{ij})_{i,j=1}^n\in\GL_n(L)$
  such that
  \begin{equation}
  \label{eq:g}
  e_j=\sum_{i=1}^nf_ig_{ij}\qquad\text{for $j=1$, \ldots, $n$.}
  \end{equation}
  The \emph{degree} $\deg\E$ is defined as
  \[
  \deg\E = 2v_\infty(\det g)\in\mathbb{Z}.
  \]
  To see that this integer does not depend on the choice of bases,
  observe that a change of bases substitutes for the matrix $g$ a
  matrix $g'$ of the form $g'=pgq$ for some $p\in\GL_n(\VR_\infty)$
  and $q\in\GL_n(\VR_U)$. We have $\det p\in\VR_S^\times$, hence
  $v_\infty(\det p)=0$. Likewise, $\det q\in\VR_U^\times$, hence
  $v(\det q)=0$ for every $F$-valuation $v$ of $L$ other than
  $v_\infty$. Since the degree of every principal divisor is zero, it
  follows that we also have $v_\infty(\det q)=0$. Therefore,
  $v_\infty(\det g)=v_\infty(\det g')$.

  A \emph{morphism} of vector bundles $(E,E_U,E_\infty)\to
  (E',E'_U,E'_\infty)$ over $C$ is an $L$-linear map $\varphi\colon
  E\to E'$ such that $\varphi(E_U)\subset E'_U$ and
  $\varphi(E_\infty)\subset E'_\infty$. When $\varphi\colon
  E\hookrightarrow E'$ is an inclusion map, the vector bundle
  $\E=(E,E_U,E_\infty)$ is said to be a \emph{subbundle} of
  $\E'=(E',E'_U,E'_\infty)$. If moreover $E_U=E\cap E'_U$ and
  $E_\infty=E\cap E'_\infty$, then the triple $(E'/E,E'_U/E_U,
  E'_\infty/E_\infty)$ is a vector bundle, which we call the
  \emph{quotient bundle} and denote by $\E'/\E$. In particular, for
  every morphism $\varphi\colon\E\to\E'$ we may consider a subbundle
  $\ker\varphi$ of $\E$ and, provided that
  $\varphi(E_U)=\varphi(E)\cap E'_U$ and
  $\varphi(E_\infty)=\varphi(E)\cap E'_\infty$, a vector bundle
  $\coker\varphi$, which is a quotient of $\E'$.
\end{defis}

\begin{example}
  \textit{Vector bundles of rank~$1$}. We use the representation of
  $C$ in Remark~\ref{rem:explicit}.
  The same arguments as in Example~\ref{ex:rk1} show that every vector
  bundle of rank~$1$ over $C$ is isomorphic to a triple
  $(L,\VR_U,(\frac xz)^n\VR_\infty)$ for some $n\in\mathbb{Z}$. The degree of
  this vector bundle is $2n$; therefore we write
  \[
  \VR_C(2n)=(L,\VR_U,\bigl(\frac xz\bigr)^n\VR_\infty).
  \]
  Note that for any $g\in L^\times$ we have as in \eqref{eq:On}
  \[
  (L,\VR_U,g\VR_\infty)\simeq \VR_C(-2v_\infty(g)).
  \]
  For the vector space of global sections we have
  \begin{align*}
  \Gamma(\VR_C(2n)) & =\{f\in\VR_U\mid v_\infty(f)\geq n\}\\
  & = \left\{f\bigl(\frac xz\bigr)+\frac yz g\bigl(\frac xz\bigr)\mid \deg
    f\leq n,\;\deg g\leq n-1\right\}. 
  \end{align*}
  Therefore,
  \begin{equation}
  \label{eq:dimGamma}
  \dim\Gamma(\VR_C(2n))=
  \begin{cases}
    2n+1 &\text{if $n\geq0$,}\\
    0 & \text{if $n<0$}.
  \end{cases}
  \end{equation}
\end{example}

We may therefore extend scalars of every vector bundle over $C$ to
get a vector bundle over $\Proj$: for any vector bundle
$\E=(E,E_U,E_\infty)$ over $C$, we define
\[
f^*(\E)=(E\otimes_FK,\: E_U\otimes_FK,\: E_\infty\otimes_FK).
\]
This $f^*(\E)$ is a vector bundle over $\Proj$ of rank $\rk
f^*(\E)=\rk\E$. If $K=F(\alpha)$, every vector in $E\otimes_FK$ has a
unique expression in the form $x\otimes1+y\otimes\alpha$ with $x$,
$y\in E$. This vector is in $E_U\otimes_FK$ (resp.\
$E_\infty\otimes_FK$) if and only if $x$, $y\in E_U$ (resp.\ $x$,
$y\in E_\infty$), hence
\begin{equation}
\label{eq:glsecext}
\Gamma\bigl(f^*(\E)\bigr) = \Gamma(\E)\otimes_FK.
\end{equation}
Since every $\VR_U$-base of $E_U$ is an $\VR_V$-base of
$E_U\otimes_FK$ and every $\VR_\infty$-base of $E_\infty$ is an
$\VR_S$-base of $E_\infty\otimes_FK$, we can compute the degree of $\E$
and the degree of $f^*(\E)$ with the same matrix $g\in \GL_n(L)$
(see~\eqref{eq:g}). We 
get $\deg\E=2v_\infty(\det g)$ and $\deg f^*(\E)=w_0(\det
g)+w_\infty(\det g)$. Because $w_0$ and $w_\infty$ are the two
valuations of $K(u)$ extending $v_\infty$, it follows that
\begin{equation}
\label{eq:degf}
\deg f^*(\E)=\deg\E.
\end{equation}
There is a construction in the opposite direction: every vector
bundle $\E'=(E',E'_V, E'_S)$ over $\Proj$ 
yields a vector bundle $f_*(\E')$ over $C$ by restriction of scalars,
i.e., by viewing $E'$ as a vector space over $L$, $E'_V$ as a module
over $\VR_U$, and $E'_S$ as a module over $\VR_\infty$. Thus, $\rk
f_*(\E')=2\rk\E'$, and
\[
\Gamma\bigl(f_*(\E')\bigr) = \Gamma(\E') \quad\text{(viewed as an
  $F$-vector space).}
\]

For the next proposition, we let $\iota$ denote the nontrivial
automorphism of $K(u)$ over $L$. For every $K(u)$-vector space $E'$,
we let ${}^\iota E'$ denote the twisted $K(u)$-vector space defined by
\[
{}^\iota E'=\{ {}^\iota x\mid x\in E'\}
\]
with the operations
\[
{}^\iota x+{}^\iota y={}^\iota(x+y) \qquad\text{and}\qquad ({}^\iota
x)\lambda = {}^\iota(x\iota(\lambda))
\]
for $x$, $y\in E'$ and $\lambda\in K(u)$. For every $\VR_V$-module
$E'_V$ and every $\VR_S$-module $E'_S$, the twisted modules ${}^\iota
E'_V$ and ${}^\iota E'_S$ are defined similarly. We may thus associate
a twisted vector bundle ${}^\iota \E'$ to every vector bundle $\E'$
over $\Proj$. Note that $\iota(u)\in u^{-1}F^\times$ (see
\eqref{eq:iotau}), hence $\iota$ interchanges 
the valuations $w_0$ and $w_\infty$. Therefore,
$w_0(\iota(\delta))+w_\infty(\iota(\delta))=w_0(\delta)+w_\infty(\delta)$
for every $\delta\in K(u)^\times$. It follows that $\deg{}^\iota
\E'=\deg\E'$; in particular, ${}^\iota\VR_{\Proj}(n)\simeq
\VR_{\Proj}(n)$ for all $n\in\mathbb{Z}$, and Grothendieck's theorem
(Theorem~\ref{thm:Gro}) yields ${}^\iota\E'\simeq\E'$ for every vector
bundle $\E'$ over $\Proj$.

\begin{prop}
  \label{prop:ff}
  \begin{enumerate}
  \item[(i)]
  For every vector bundle $\E$ over $C$, we have
  \[
  f_*f^*(\E)\simeq\E\oplus\E.
  \]
  \item[(ii)]
  For every vector bundle $\E'$ over $\Proj$, we have a canonical
  isomorphism
  \[
  f^*f_*(\E')\simeq \E'\oplus{}^\iota \E',
  \]
  and an isomorphism $f^*f_*(\E')\simeq \E'\oplus \E'$.
  \end{enumerate}
\end{prop}

\begin{proof}
  (i) Let $\alpha\in K$ be such that $K=F(\alpha)$. For every
  $L$-vector space $E$, mapping $x\otimes1+y\otimes\alpha$ to $(x,y)$
  for $x$, $y\in E$ defines an $L$-linear isomorphism
  $E\otimes_FK\stackrel{\sim}{\to} E\oplus E$. We thus get an
  isomorphism $f_*f^*(\E)\simeq\E\oplus\E$.

  (ii) For every $K(u)$-vector space $E'$, we identify $E'\otimes_FK$
  with $E'\otimes{}^\iota E'$ by mapping $x\otimes\lambda$ to
  $(x\lambda, ({}^\iota x)\lambda)$. We thus get a canonical
  isomorphism $f^*f_*(\E')\simeq\E'\oplus{}^\iota\E'$.
\end{proof}

\begin{cor}
  \label{cor:degf}
  For every vector bundle $\E'$ over $\Proj$,
  \[
  \deg f_*(\E')=2\deg\E'.
  \]
\end{cor}

\begin{proof}
  Proposition~\ref{prop:ff}(ii) and \eqref{eq:degf} yield
  \[
  \deg f_*(\E')=\deg(\E'\oplus\E') = 2\deg\E'.
  \qedhere
  \]
\end{proof}

\begin{cor}
  \label{cor:fO}
  For every $n\in\mathbb{Z}$ we have
  \begin{enumerate}
  \item[(i)]
  $f^*\bigl(\VR_C(2n)\bigr) \simeq \VR_{\Proj}(2n)$,
  \item[(ii)]
  $f_*\bigl(\VR_{\Proj}(2n)\bigr)\simeq \VR_C(2n)\oplus \VR_C(2n)$.
  \end{enumerate}
  Moreover, $f_*\bigl(\VR_{\Proj}(2n+1)\bigr)$ is an indecomposable
  vector bundle of rank~$2$ and degree~$4n+2$ over $C$.
\end{cor}

\begin{proof}
  From the definitions of $\VR_C(2n)$ and $f^*$, we have
  \[
  f^*\bigl(\VR_C(2n)\bigr) = (K(u),\VR_V,t^n\VR_S).
  \]
  By~\eqref{eq:On} it follows that
  \[
  f^*\bigl(\VR_C(2n)\bigr) \simeq \VR_{\Proj}(-w_0(t^n)-w_\infty(t^n))
  = \VR_{\Proj}(2n).
  \]
  This proves~(i). Moreover, applying $f_*$ to each side, we get
  \[
  f_*\bigl(\VR_{\Proj}(2n)\bigr)\simeq f_*f^*\bigl(\VR_C(2n)\bigr),
  \]
  and (ii) follows from Proposition~\ref{prop:ff}(i).

  By definition, it is clear that $f_*\bigl(\VR_{\Proj}(2n+1)\bigr)$
  is a vector bundle of rank~$2$. Corollary~\ref{cor:degf} shows that
  its degree is $4n+2$, and it only remains to show that this vector
  bundle is indecomposable. Any nontrivial decomposition involves two
  vector bundles of rank~$1$, and has therefore the form
  \[
  f_*\bigl(\VR_{\Proj}(2n+1)\bigr) \simeq \VR_C(2m_1) \oplus
  \VR_C(2m_2)
  \]
  for some $m_1$, $m_2\in\mathbb{Z}$. By applying $f^*$ to each side
  and using~(i) and Proposition~\ref{prop:ff}(ii), we obtain
  \[
  \VR_{\Proj}(2n+1)\oplus \VR_{\Proj}(2n+1) \simeq \VR_{\Proj}(2m_1)
  \oplus \VR_{\Proj}(2m_2).
  \]
  This is a contradiction because the Grothendieck decomposition in
  Theorem~\ref{thm:Gro} is unique up to permutation of the summands.
\end{proof}

We write $\Indec_C(4n+2)=f_*\bigl(\VR_{\Proj}(2n+1)\bigr)$. In the rest of
this section, our goal is to prove that every vector bundle over $C$
decomposes in a unique way in a direct sum of vector bundles of the
form $\VR_C(2n)$ and $\Indec_C(4n+2)$.

\begin{prop}
  \label{prop:dimglsecfinite}
  For every vector bundle $\E$ over $C$, the space of global sections
  $\Gamma(\E)$ is finite-dimensional.
\end{prop}

\begin{proof}
  This readily follows from~\eqref{eq:glsecext} and
  Corollary~\ref{cor:glsecP}. 
\end{proof}

\begin{cor}
  \label{cor:EndB}
  For every vector bundle $\E$ over $C$, the $F$-algebra $\End\E$ is
  finite-dimensional. Moreover, the idempotents in $\End\E$ split: any
  idempotent $e\in\End\E$ yields a decomposition $\E=\ker e \oplus
  \image e$. If $\E$ does not decompose into a sum of nontrivial vector
  bundles, then $\End\E$ is a local ring (i.e., the noninvertible
  elements form an ideal).
\end{cor}

\begin{proof}
  For $\E=(E,E_U,E_\infty)$, we have $\End\E=\Gamma(\calEnd\E)$ where
  \[
  \calEnd\E = (\End_LE,\: \End_{\VR_U}E_U,\:
  \End_{\VR_\infty}E_\infty).
  \]
  Therefore, Proposition~\ref{prop:dimglsecfinite} shows that the
  dimension of $\End\E$ is finite. This algebra is therefore right
  (and left) Artinian. If $e\in\End\E$ is an idempotent, then for
  every vector $x\in E$ we have $x=\bigl(x-e(x)\bigr)+e(x)$, hence
  \[
  E=\ker e\oplus\image e,\qquad E_U=(E_U\cap\ker e)\oplus
  (E_U\cap\image e),
  \]
  \[
  E_\infty=(E_\infty\cap\ker e)\oplus
  (E_\infty\cap\image e).
  \]
  This shows that $e$ splits. If $\E$ is indecomposable, then $\End\E$
  has no nontrivial idempotents. It follows from Lam
  \cite[Cor.~(19.19)]{Lam} that $\End\E$ is a local ring.
\end{proof}

The properties of $\End\E$ established in Corollary~\ref{cor:EndB}
allow us to use the general approach to the Krull--Schmidt theorem in
Bass \cite[Ch.~I, (3.6)]{Bass} (see also Lam \cite[(19.21)]{Lam}) to
derive the following ``Krull--Schmidt'' result:

\begin{cor}
  \label{cor:KS}
  Every vector bundle over $C$ decomposes into a sum of
  indecomposable vector bundles, and the decomposition is unique up
  to isomorphism and the order of summands.
\end{cor}

Note that the existence of a decomposition into indecomposable vector
bundles is clear by induction on the rank.

\begin{thm}
  \label{thm:classbnd}
  Every vector bundle $\E$ over $C$ has a decomposition of the form
  \[
  \E\simeq \VR_C(2k_1)\oplus\cdots \oplus\VR_C(2k_r) \oplus
  \Indec_C(4\ell_1+2) \oplus\cdots\oplus \Indec_C(4\ell_m+2)
  \]
  for some $k_1$, \ldots, $k_r$, $\ell_1$, \ldots,
  $\ell_m\in\mathbb{Z}$. The sequences $(k_1,\ldots, k_r)$ and
  $(\ell_1,\ldots,\ell_m)$ are uniquely determined by $\E$ up to
  permutation of the entries.
\end{thm}

\begin{proof}
  In view of Corollary~\ref{cor:KS}, it only remains to show that the
  vector bundles $\VR_C(2k)$ and $\Indec_C(4\ell+2)$ are the only
  indecomposable vector bundles over $C$ up to isomorphism. Suppose
  $\E$ is an indecomposable vector bundle over $C$. Grothendieck's
  theorem (Theorem~\ref{thm:Gro}) yields integers $n_1$, \ldots,
  $n_p\in\mathbb{Z}$ such that
  \[
  f^*(\E)\simeq \VR_{\Proj}(n_1)\oplus\cdots\oplus \VR_{\Proj}(n_p).
  \]
  Applying $f_*$ to each side, we get by Proposition~\ref{prop:ff}(i)
  \[
  \E\oplus\E \simeq f_*\bigl(\VR_{\Proj}(n_1)\bigr) \oplus \cdots\oplus
  f_*\bigl(\VR_{\Proj}(n_p)\bigr).
  \]
  If $n_1$ is even, then $f_*\bigl(\VR_{\Proj}(n_1)\bigr)\simeq
  \VR_C(n_1)\oplus \VR_C(n_1)$ by Corollary~\ref{cor:fO}, hence $p=1$
  and $\E\simeq\VR_C(n_1)$. If $n_1$ is odd, then
  $f_*\bigl(\VR_{\Proj}(n_1)\bigr)$ is indecomposable by
  Corollary~\ref{cor:fO}, hence we must have $\E\simeq
  f_*\bigl(\VR_{\Proj}(n_1)\bigr)=\Indec_C(2n_1)$ (and $p=2$, and $n_2=n_1$).
\end{proof}

\begin{example}
  \textit{The tautological vector bundle.}
  We use the representation of $C$ in Remark~\ref{rem:explicit}. Let
  \[
  \Q_C=\VR_C(0)\otimes_FQ=(Q_L,Q_U,Q_\infty)
  \]
  where $Q_L=L\otimes_FQ$, $Q_U=\VR_U\otimes_FQ$,
  $Q_\infty=\VR_\infty\otimes_FQ$. 
  Consider the element 
  \[
  e:=\frac{x}{z}r+\frac{y}{z}s+t\in Q_L
  \]
  and the $2$-dimensional right ideal $E=eQ_L$. We define the bundle
  $\T=(E, E_U, E_{\infty})$ by
  \[
  E_U=E\cap Q_U \quad \text{and} \quad E_{\infty}=E\cap Q_{\infty}.
  \]
\begin{lemma}
We have
\begin{enumerate}
  \item[(a)] $E_U=eQ\cdot\VR_U=er\VR_U \oplus es\VR_U$,
  \item[(b)] $E_{\infty}=e\frac
    zyQ\cdot\VR_\infty=e\frac{z}{y}r\VR_{\infty}\oplus 
    e\frac{z}{y}t\VR_{\infty}$. 
\end{enumerate}
\end{lemma}

\begin{proof}
We first note that
\begin{equation}\label{main}
e\frac{x}{z}r+ e\frac{y}{z}s +et=e^2=0.
\end{equation}
Since $er\VR_U+es\VR_U\subset E_U$, to prove~(a) it suffices to show
$E_U\subset eQ\cdot\VR_U$ and $eQ\subset er\VR_U+es\VR_U$. We start
with the second inclusion.

It follows from \eqref{main} that 
\begin{equation}\label{1}
et=-e\frac{x}{z}r- e\frac{y}{z}s \in er\VR_U + es\VR_U.
\end{equation}
Write $\ell:=rs\in Q$. Note that $\ell\notin F$ and
$(rF+sF)\ell=rF+sF$. Multiplying \eqref{1} by $\ell$ on the right, we
then get 
\begin{equation}\label{12}
et\ell=-e\frac{x}{z}r\ell- e\frac{y}{z}s\ell \in er\ell\VR_U+
es\ell\VR_U=er\VR_U+ es\VR_U. 
\end{equation}
Also $t\ell\notin V$: for if $t\ell\in V$ then $V\ell=V$, hence $\ell$
lies in the orthogonal of $V$ for the bilinear form $\Trd_Q(XY)$; it
follows that $\ell\in F$, a contradiction. Therefore, $(r,s,t,t\ell)$
is a base of $Q$. The inclusion $eQ\subset er\VR_U+es\VR_U$ follows
from~\eqref{1} and \eqref{12}.  

We next show $E_U\subset eQ\cdot\VR_U$. Equations~\eqref{1} and
\eqref{12} show that $eQ_L$ is spanned by $er$ and $es$, hence every
element $\xi\in E_U$ has the form $\xi=er\lambda+es\mu$ for some
$\lambda$, $\mu\in L$. We show that the hypothesis $\xi\in
Q_U$ implies $\lambda$, $\mu\in\VR_U$. Let $\invo$ denote the
quaternion conjugation. Since $\xi\in Q_U$, we have $\xi
s-s\overline\xi\in Q_U$. Computation yields
\[
\xi s-s\overline\xi= (ers-sre)\lambda=(trs-srt)\lambda.
\]
By the choice of $t$ we have $b_q(t,r)=b_q(t,s)=0$, hence $t$
anticommutes with $r$ and $s$, and therefore
\[
\xi s-s\overline\xi=(rs-sr)t\lambda.
\]
Since $rs-sr\neq0$ and $\xi s-s\overline\xi\in Q_U$, it follows that
$\lambda\in\VR_U$. Therefore, $es\mu=\xi-er\lambda\in Q_U$, hence
$e\mu\in Q_U$. It follows that $\mu\in\VR_U$, because $e\mu=r\frac
xz\mu+s\frac yz\mu+t\mu$. The proof of~(a) is thus complete.

The proof of (b) is similar. Since $e\frac zyr\VR_\infty+e\frac
zyt\VR_\infty\subset E_\infty$, it suffices to prove $E_\infty\subset
e\frac zyQ\cdot\VR_\infty$ and $eQ\subset
er\VR_\infty+es\VR_\infty$. We again start with the second inclusion.

It follows from \eqref{main} that 
\begin{equation}\label{2}
es=-e\frac{x}{y}r- e\frac{z}{y}t \in er\VR_{\infty} + et\VR_{\infty}.
\end{equation}
Write $m:=rt\in Q$. Note that $m\notin F$ and
$(rF+tF)m=rF+tF$. Multiplying \eqref{2} by $m$ on the right, we then get
\begin{equation}\label{22}
esm=-e\frac{x}{y}rm- e\frac{z}{y}tm \in erm\VR_{\infty} +
etm\VR_{\infty} =er\VR_{\infty} + et\VR_{\infty}. 
\end{equation}
Also $sm\notin V$ since $Vm\neq V$. Therefore, $(r,s,t,sm)$ is a base
of $Q$. The inclusion $eQ\subset er\VR_\infty+es\VR_\infty$ follows
from~\eqref{2} and \eqref{22}.

It also follows from~\eqref{2} and \eqref{22} that $eQ_L$ is spanned
by $e\frac zyr$ and $e\frac zyt$, hence every element $\xi\in
E_\infty$ has the form $\xi=e\frac zyr\lambda+e\frac zyt\mu$ for some
$\lambda$, $\mu\in L$. We show that $\xi\in Q_\infty$ implies
$\lambda$, $\mu\in\VR_\infty$. Since $t$ anticommutes with $r$ and
$s$, we have
\[
\xi t-t\overline\xi=(ert-tre)\frac zy\lambda=(sr-rs)t\lambda.
\]
Because $\xi t-t\overline\xi\in Q_\infty$, it follows that
$\lambda\in\VR_\infty$. Then $\xi-e\frac zyr\lambda=e\frac zyt\mu\in
Q_\infty$, and it follows that $\mu\in\VR_\infty$.
\end{proof}

It follows from \eqref{2} that the change of base matrix between the
bases $(er, es)$ and $(e\frac{z}{y}r, e\frac{z}{y}t)$ is equal to 
\[
\begin{pmatrix}
         \frac{y}{z} & -\frac{x}{z} \\
         0 & -1 \\
\end{pmatrix}.
\]
  Therefore, $\deg\T=2v_\infty(\frac yz)=-2$. Note also that 
  $\Gamma(\mathcal{T})=\{0\}$ because $E_U\cap E_\infty=E\cap Q$ and $Q$ is a
  division algebra. Therefore, $\mathcal{T}$ is indecomposable because if
  $\mathcal{T}\simeq\VR_C(2m)\oplus\VR_C(2p)$ for some $m$,
  $p\in\mathbb{Z}$ 
  then comparing the degrees we see that $m+p=-1$. But then one of
  $m$, $p$ must be nonnegative, and then $\VR_C(2m)$ or $\VR_C(2p)$
  has nonzero global sections. Thus, we must have $\mathcal{T}\simeq
  \Indec_C(-2)$. 
\end{example}

Note that $Q$ acts naturally on the bundle $\T$, i.e., $\T$ is a
$Q$-module bundle, so we have a canonical embedding
$Q^{\operatorname{op}}\hookrightarrow\End\T$. In fact, since
$\T\simeq\Indec_C(-2)$ we have by Corollary~\ref{cor:dual} and
\eqref{eq:IndecIndec}
\[
\calEnd(\T)\simeq\T\otimes\T\spcheck\simeq \Indec_C(-2)\otimes\Indec_C(2)
\simeq\VR_C(0)^{\oplus4}.
\]
Therefore, $\dim\End\T=4$, hence
\begin{equation*}
  \End\T\simeq Q^{\operatorname{op}}\simeq Q.
\end{equation*}
Since $\Indec_C(2n)=\Indec_C(-2)\otimes\VR_C(n+1)$ for all odd $n$
(see~\eqref{eq:IndecO}), we also have
\begin{equation}
  \label{eq:EndIndec2}
  \End\bigl(\Indec_C(2n)\bigr)\simeq Q \qquad\text{for all odd $n$.}
\end{equation}

\subsection{Duality}

The \emph{dual} of a vector bundle $\E=(E,E_U,E_\infty)$ over $C$ is
the vector bundle
\[
\E\spcheck=(\Hom_L(E,L),\;\Hom_{\VR_U}(E_U,\VR_U),\:
\Hom_{\VR_\infty}(E_\infty,\VR_\infty)).
\]

\begin{prop}
  \label{prop:degdual}
  $\deg\E\spcheck = -\deg\E$.
\end{prop}

\begin{proof}
  Let $(e_i)_{i=1}^n$ be an $\VR_U$-base of $E_U$ and $(f_i)_{i=1}^n$
  be an $\VR_\infty$-base of $E_\infty$, and let
  $g=(g_{ij})_{i,j=1}^n\in\GL_n(L)$ be defined by the equations
  \[
  e_j=\sum_{i=1}^n f_ig_{ij}\qquad\text{for $j=1$, \ldots, $n$}.
  \]
  So, by definition, $\deg\E=2v_\infty(\det g)$. The dual bases
  $(e_i^*)_{i=1}^n$ and $(f_i^*)_{i=1}^n$ are bases of
  $\Hom_{\VR_U}(E_U,\VR_U)$ and
  $\Hom_{\VR_\infty}(E_\infty,\VR_\infty)$ respectively, and they are
  related by
  \[
  e_j^*=\sum_{i=1}^n f_i^*g'_{ij}\qquad\text{for $j=1$, \ldots, $n$},
  \]
  where the matrix $g'=(g'_{ij})_{i,j=1}^n$ is
  $(g^t)^{-1}$. Therefore, $\det g'=(\det g)^{-1}$ and $\deg\E\spcheck
  = - \deg\E$. 
\end{proof}

\begin{cor}
  \label{cor:dual}
  If $\E\simeq \VR_C(2k_1)\oplus\cdots \oplus\VR_C(2k_r) \oplus
  \Indec_C(4\ell_1+2) \oplus\cdots\oplus \Indec_C(4\ell_m+2)$
  for some $k_1$, \ldots, $k_r$, $\ell_1$, \ldots,
  $\ell_m\in\mathbb{Z}$, then 
  \[
  \E\spcheck \simeq\VR_C(-2k_1)\oplus\cdots \oplus\VR_C(-2k_r) \oplus
  \Indec_C(-4\ell_1-2) \oplus\cdots\oplus \Indec_C(-4\ell_m-2).
  \]
\end{cor}

\begin{proof}
  $\VR_C(2k)\spcheck$ is a vector bundle of rank~$1$ and degree~$-2k$,
  hence $\VR(2k)\spcheck\simeq\VR_C(-2k)$. Similarly,
  $\Indec_C(4\ell+2)\spcheck$ is an indecomposable vector bundle of
  rank~$2$ and degree~$-4\ell-2$, hence $\Indec_C(4\ell+2)\spcheck
  \simeq \Indec_C(-4\ell-2)$.
\end{proof}

\end{document}